\documentclass[submission,copyright,creativecommons,sharealike,noncommercial]{eptcs}
\providecommand{\event}{QPL 2018}
\pdfoutput=1

\usepackage[toc,page]{appendix}
\usepackage{mathtools} 
\usepackage{amssymb} 
\usepackage{bbold} 
\usepackage{amsthm} 
\usepackage{stmaryrd} 

\usepackage{relsize} 
\usepackage{microtype} 
\usepackage{multicol} 
\usepackage{csquotes} 

\usepackage[nocompress]{cite} 
\usepackage[hyperpageref]{backref} 
\renewcommand*{\backref}[1]{}
\renewcommand*{\backrefalt}[4]{%
	\ifcase #1 (Not cited.)%
	\or        (Cited on page~#2.)%
	\else      (Cited on pages~#2.)%
	\fi}

\usepackage{graphicx} 
\usepackage{subcaption} 
\usepackage{wrapfig} 
\usepackage[usenames,dvipsnames]{xcolor} 
\usepackage{tikz} 
\usetikzlibrary{
	math,
	decorations.markings,
	positioning,
	arrows,
	shapes,
	automata,
	petri,
	decorations,
	backgrounds,
	calc,
	fit,
	quotes
	}

\usepackage{thmtools}
\usepackage{thm-restate}

		\newcounter{theorem_c} 
		\numberwithin{theorem_c}{section} 
		\numberwithin{equation}{section} 

		\theoremstyle{plain}

		\newtheorem{lemma}[theorem_c]{Lemma}

		\newtheoremstyle{exampstyle}
		  {2mm} 
		  {2mm} 
		  {\itshape} 
		  {} 
		  {\bfseries} 
		  {.} 
		  {.5em} 
		  {} 

		\theoremstyle{exampstyle}
		\newtheorem{definition}[theorem_c]{Definition}
		
		\newtheorem{remark}[theorem_c]{Remark}

	\newcommand{\naturals}{\mathbb{N}} 
	\newcommand{\integers}{\mathbb{Z}} 
	\newcommand{\inject}{\hookrightarrow} 

	\newcommand{\eqdef}{:=} 
	\newcommand{\suchthat}[2]{\left\{#1 \: \middle\vert \: #2\right\}} 

		\newcommand{\id}[1]{id_{#1}} 

		\newcommand{\tensor}{\otimes} 
		\newcommand{\bigtensor}{\bigotimes} 
		\newcommand{\tensorUnit}{I} 
		\newcommand{\cartesianTensor}{\times} 

		\newcommand{\Hom}[3]{\operatorname{Hom}_{\,#1}\left[#2,#3\right]} 
		\newcommand{\Nat}[3]{\operatorname{Nat}_{#1}[#2,#3]} 

		\newcommand{\Pin}[1]{{^\circ}(#1)} 
		\newcommand{\Pout}[1]{{(#1)^\circ}} 
		\newcommand{\Mset}[1]{\mathcal{M}^{\naturals}_{#1}} 
		\newcommand{\Iset}[1]{\mathcal{M}^{\integers}_{#1}} 
		\newcommand{\IStrings}[1]{#1^{\integers}} 
		
		\newcommand{\Gr}[1]{\mathcal{G}_{#1}} 
		\newcommand{\Multiplicity}{\mathfrak{M}} 
		
		\newcommand{\Net}[1]{(P_{#1},T_{#1},\Pin{-}_{#1},\Pout{-}_{#1})}
		
		\newcommand{\Fold}[1]{\mathfrak{F}(#1)} 
		\newcommand{\UnFold}[1]{\mathfrak{U}(#1)} 

		\newcommand{\PetriZ}{\textbf{Petri}^\integers}
		\newcommand{\PetriZZ}{\textbf{Petri}^{\integers \text{state}}}
		\newcommand{\PetriN}{\textbf{Petri}^\naturals}
		\newcommand{\GExPetri}{\textbf{GExPetri}}
		\newcommand{\ExPetriZ}{\textbf{ExPetri}^\integers}
		\newcommand{\ExPetriZZ}{\textbf{ExPetri}^{\integers \text{state}}}
		\newcommand{\ExPetriN}{\textbf{ExPetri}^\naturals}

		\newcommand{\CategoryC}{\mathcal{C}}
		\newcommand{\CategoryD}{\mathcal{D}}
		\newcommand{\CategoryE}{\mathcal{E}}

		\newcommand{\obj}[1]{\operatorname{obj} \, #1} 

		\tikzset{
			oriented WD/.style={
				every to/.style={out=0,in=180,draw},
				label/.style={
					font=\everymath\expandafter{\the\everymath\scriptstyle},
					inner sep=0pt,
					node distance=2pt and -2pt},
				semithick,
				node distance=1 and 1,
				decoration={markings, mark=at position \stringdecpos with \stringdec},
				ar/.style={postaction={decorate}},
				execute at begin picture={\tikzset{
						x=\bbx, y=\bby,
						every fit/.style={inner xsep=\bbx, inner ysep=\bby}}}
			},
			string decoration/.store in=\stringdec,
			string decoration={\arrow{stealth};},
			string decoration pos/.store in=\stringdecpos,
			string decoration pos=.7,
			bbx/.store in=\bbx,
			bbx = 1.5cm,
			bby/.store in=\bby,
			bby = 1.5ex,
			bb port sep/.store in=\bbportsep,
			bb port sep=1.5,
			bb port length/.store in=\bbportlen,
			bb port length=4pt,
			bb penetrate/.store in=\bbpenetrate,
			bb penetrate=0,
			bb min width/.store in=\bbminwidth,
			bb min width=1cm,
			bb rounded corners/.store in=\bbcorners,
			bb rounded corners=2pt,
			bb small/.style={bb port sep=1, bb port length=2.5pt, bbx=.4cm, bb min width=.4cm, 
				bby=.7ex},
			bb medium/.style={bb port sep=1, bb port length=2.5pt, bbx=.4cm, bb min width=.4cm, 
				bby=.9ex},
			bb/.code 2 args={
				\pgfmathsetlengthmacro{\bbheight}{\bbportsep * (max(#1,#2)+1) * \bby}
				\pgfkeysalso{draw,minimum height=\bbheight,minimum width=\bbminwidth,outer 
					sep=0pt,
					rounded corners=\bbcorners,thick,
					prefix after command={\pgfextra{\let\fixname\tikzlastnode}},
					append after command={\pgfextra{\draw
							\ifnum #1=0{} \else foreach \i in {1,...,#1} {
								($(\fixname.north west)!{\i/(#1+1)}!(\fixname.south west)$) +(-
								\bbportlen,0) 
								coordinate (\fixname_in\i) -- +(\bbpenetrate,0) coordinate (\fixname_in\i')}\fi 
							\ifnum #2=0{} \else foreach \i in {1,...,#2} {
								($(\fixname.north east)!{\i/(#2+1)}!(\fixname.south east)$) +(-
								\bbpenetrate,0) 
								coordinate (\fixname_out\i') -- +(\bbportlen,0) coordinate (\fixname_out\i)}\fi;
				}}}
			},
			bb name/.style={append after command={\pgfextra{\node[anchor=north] at 
						(\fixname.north) {#1};}}}
		}

		\tikzset{
			unoriented WD/.style={
				every to/.style={draw},
				shorten <=-\penetration, shorten >=-\penetration,
				label distance=-2pt,
				thick,
				node distance=\spacing,
				execute at begin picture={\tikzset{
						x=\spacing, y=\spacing}}
			},
			pack size/.store in=\psize,
			pack size = 8pt,
			spacing/.store in=\spacing,
			spacing = 8pt,
			link size/.store in=\lsize,
			link size = 2pt,
			penetration/.store in=\penetration,
			penetration = 2pt,
			pack color/.store in=\pcolor,
			pack color = blue,
			pack inside color/.store in=\picolor,
			pack inside color=blue!20,
			pack outside color/.store in=\pocolor,
			pack outside color=blue!50!black,
			surround sep/.store in=\ssep,
			surround sep=8pt,
			link/.style={
				circle, 
				draw=black, 
				fill=black,
				inner sep=0pt, 
				minimum size=\lsize
			},
			pack/.style={
				circle, 
				draw = \pocolor, 
				fill = \picolor,
				inner sep = .25*\psize,
				minimum size = \psize
			},
			outer pack/.style={
				ellipse, 
				draw,
				inner sep=\ssep,
				color=\pocolor,
			},
			intermediate pack/.style={
				ellipse,
				dashed, 
				draw,
				inner sep=\ssep,
				color=\pocolor,
			},
		}

		\tikzset{Yonepart/.pic={
				\node[bb={1}{2},bb name = {\tiny$X_{11}$}] (X11) {};
				\node[bb={2}{2},below right=of X11,bb name = {\tiny$X_{12}$}] (X12) {};
				\node[bb={2}{1}, above right=of X12,bb name = {\tiny$X_{13}$}] (X13) {};
				\node[bb={2}{2}, fit={($(X11.north west)+(.3,1.5)$) (X12)  ($(X13.east)+(-.3,0)$)},bb name = {\scriptsize $Y_1$}] (Y1) {};
				\draw (Y1_in1') to (X11_in1);	
				\draw (Y1_in2') to (X12_in2);
				\draw (X11_out1) to (X13_in1);
				\draw (X11_out2) to (X12_in1);
				\draw (X12_out1) to (X13_in2);
				\draw (X12_out2) to (Y1_out2');
				\draw (X13_out1) to (Y1_out1');
				\coordinate (bottombox) at ($(X12.south)$);
				\coordinate (rightbox) at ($(X13.east)$);
				\coordinate (Y1northwest) at ($(Y1.north west)$);
			}
		}
		
		\tikzset{Ytwopart/.pic={
				\node[bb={2}{2}, bb name = {\tiny$X_{21}$}] (X21) {};
				\node[bb={1}{2},above right=-1 and 1 of X21,bb name = {\tiny$X_{22}$}] (X22) {};
				\node[bb={1}{2}, fit={($(X21.south west)+(-.25,0)$) ($(X22.north east)+(.25,3.5)$)},bb name = {\scriptsize$Y_2$}] (Y2){};
				\draw (Y2_in1') to (X21_in2);
				\draw (X21_out1) to (X22_in1);
				\draw (X22_out2) to (Y2_out1');
				\draw let \p1=(X22.south east), \p2=($(Y2_out2)$), \n1={\y1-\bby}, \n2=\bbportlen in
				(X21_out2) to (\x1+\n2,\n1) -- (\x1+\n2,\n1) to (Y2_out2');
				\draw let \p1=(X22.north east), \p2=(X21.north west), \n1={\y1+\bby}, \n2=\bbportlen in
				(X22_out1) to[in=0] (\x1+\n2,\n1) -- (\x2-\n2,\n1) to[out=180] (X21_in1);
			}
		}
		
		\tikzset{SmallNeuronPic/.pic={
				\node[bb={3}{1}] (N1) {$\scriptstyle N_1$};
				\node[bb={2}{1}, above right=.7 and 3.5 of N1] (N2) {$\scriptstyle N_2$};
				\node[bb={2}{1}, below =of N2] (N3) {$\scriptstyle N_3$};
				\node[bb={3}{1}, below =of N3] (N4) {$\scriptstyle N_4$};
				\node[bb={6}{8}, fit={($(N1.west)-(.5,0)$) ($(N2.north)+(0,2)$) ($(N3.east)+(1.5,0)$) ($(N4.south)-(0,1)$)}, bb name={$\scriptstyle X$}] (X) {};
				\draw (X_in1') to (N2_in1);
				\draw (X_in2') to (N1_in1);
				\draw (X_in3') to (N1_in2);
				\draw (X_in4') to (N1_in3);
				\draw (X_in6') to (N4_in2);
				\draw (N1_out1) to (N2_in2);
				\draw (N1_out1) to (N3_in1);
				\draw (N1_out1) to (N4_in1);
				\draw (N2_out1) to (X_out1');
				\draw (N2_out1) to (X_out2');
				\draw (N2_out1) to (X_out3');
				\draw (N3_out1) to (X_out4');
				\draw (N3_out1) to (X_out5');
				\draw (N3_out1) to (X_out6');
				\draw (N4_out1) to (X_out7');
				\draw (N4_out1) to (X_out8'); 
				\draw (X_in5') to[looseness=2] (N3_in2);
				\draw let \p1=(N4.south east), \p2=(N4.south west), \n1={\y2-\bby}, \n2=\bbportlen in
				(N3_out1) to[in=0] (\x1+\n2,\n1) -- (\x2-\n2,\n1) to[out=180] (N4_in3);
			}
		}
		
		\tikzset{SmallNeuronDashed/.pic={
				\node[bb={3}{1}] (N1) {$\scriptstyle N_1$};
				\node[bb={2}{1}, above right=.7 and 3.5 of N1] (N2) {$\scriptstyle N_2$};
				\node[bb={2}{1}, below =of N2] (N3) {$\scriptstyle N_3$};
				\node[bb={3}{1}, below =of N3] (N4) {$\scriptstyle N_4$};
				\node[bb={6}{8}, fit={($(N1.west)-(.5,0)$) ($(N2.north)+(0,2)$) ($(N3.east)+(1.5,0)$) ($(N4.south)-(0,1)$)}, bb name={$\scriptstyle X$}] (X) {};
				\draw[dashed] (X_in1') to (N2_in1);
				\draw[dashed] (X_in2') to (N1_in1);
				\draw[dashed] (X_in3') to (N1_in2);
				\draw[dashed] (X_in4') to (N1_in3);
				\draw[dashed] (X_in6') to (N4_in2);
				\draw[dashed] (N1_out1) to (N2_in2);
				\draw[dashed] (N1_out1) to (N3_in1);
				\draw[dashed] (N1_out1) to (N4_in1);
				\draw[dashed] (N2_out1) to (X_out1');
				\draw[dashed] (N2_out1) to (X_out2');
				\draw[dashed] (N2_out1) to (X_out3');
				\draw[dashed] (N3_out1) to (X_out4');
				\draw[dashed] (N3_out1) to (X_out5');
				\draw[dashed] (N3_out1) to (X_out6');
				\draw[dashed] (N4_out1) to (X_out7');
				\draw[dashed] (N4_out1) to (X_out8'); 
				\draw[dashed] (X_in5') to[looseness=2] (N3_in2);
				\draw[dashed] let \p1=(N4.south east), \p2=(N4.south west), \n1={\y2-\bby}, \n2=\bbportlen in
				(N3_out1) to[in=0] (\x1+\n2,\n1) -- (\x2-\n2,\n1) to[out=180] (N4_in3);
			}
		}
		
		\tikzset{SmallNestingPic/.pic={
				\path (0,0) pic [purple] {Yonepart};
				\path ($(rightbox)+(5,-5)$) pic [orange] {Ytwopart};
				
				\node[bb={1}{2}, fit={($(Y1northwest)+(-.5,4)$) ($(Y2.south east)+(1,0)$)}, bb name={\small $Z$}] (Z) {};
				\draw (Z_in1') to (Y1_in2);
				\draw let \p1=(Y2.north west),\p2=(Y2.north east),\n1={\y2+\bby},\n2=\bbportlen in
				(Y1_out1) to (\x1+\n2,\n1)--(\x2+\n2,\n1) to (Z_out1');
				\draw (Y1_out2) to (Y2_in1);
				\draw (Y2_out2) to (Z_out2');
				\draw let \p1=(Y2.north east), \p2=(Y1.north west), \n1={\y2+\bby}, \n2=\bbportlen in
				(Y2_out1) to[in=0] (\x1+\n2,\n1) -- (\x2-\n2,\n1) to[out=180] (Y1_in1);
			}
		}

		\tikzset{Zredgreen/.pic={
				\node[bb={2}{2}, green!50!black, bb name = $\scriptstyle Y_1$] (YY1) {};
				\node[bb={1}{2}, red, below right=-1 and 2 of YY1, bb name=$\scriptstyle Y_2$] (YY2) {};
				\node[bb={1}{2}, fit={($(YY1.north west)+(-.5,4)$) ($(YY2.south east)+(.5,-2)$)}, bb name={\scriptsize $Z$}] (Z) {};
				\draw (Z_in1') to (YY1_in2);
				\draw (YY1_out1) to (Z_out1');
				\draw (YY1_out2) to (YY2_in1);
				\draw (YY2_out2) to (Z_out2');
				\draw let \p1=(YY2.north east), \p2=(YY1.north west), \n1={\y2+\bby}, \n2=\bbportlen in
				(YY2_out1) to[in=0] (\x1+\n2,\n1) -- (\x2-\n2,\n1) to[out=180] (YY1_in1);
			}
		}
		
		\tikzset{Zcombined/.pic={
				\node[bb={1}{2},green!25!black,bb name = {\tiny$X_{11}$}] (X11) {};
				\node[bb={2}{2},green!25!black,below right=of X11,bb name = {\tiny$X_{12}$}] (X12) {};
				\node[bb={2}{1}, green!25!black,above right=of X12,bb name = {\tiny$X_{13}$}] (X13) {};
				\draw (X11_out1) to (X13_in1);
				\draw (X11_out2) to (X12_in1);
				\draw (X12_out1) to (X13_in2);
				
				\node[bb={2}{2}, red!30!black, below right = 0 and 1.25 of X12, bb name = {\tiny$X_{21}$}] (X21) {};
				\node[bb={1}{2}, red!30!black, above right=-1 and 1 of X21,bb name = {\tiny$X_{22}$}] (X22) {};
				\draw (X21_out1) to (X22_in1);
				\draw let \p1=(X22.north east), \p2=(X21.north west), \n1={\y1+\bby}, \n2=\bbportlen in
				(X22_out1) to[in=0] (\x1+\n2,\n1) -- (\x2-\n2,\n1) to[out=180] (X21_in1);
				
				\node[bb={1}{2}, fit = {($(X11.north east)+(-1,3)$) (X12) (X13) ($(X21.south)+(0,-1)$) ($(X22.east)+(.5,0)$)}, bb name ={\scriptsize $Z$}] (Z) {};
				
				\draw (Z_in1') to (X12_in2);
				\draw (X13_out1) to (Z_out1');
				\draw (X12_out2) to (X21_in2);
				\draw let \p1=(X22.south east),\n1={\y1-\bby}, \n2=\bbportlen in
				(X21_out2) to (\x1+\n2,\n1) to (Z_out2');
				\draw let \p1=(X22.north east), \p2=(X11.north west), \n1={\y2+\bby}, \n2=\bbportlen in
				(X22_out2) to[in=0] (\x1+\n2,\n1) -- (\x2-\n2,\n1) to[out=180] (X11_in1);
			}
		}

\tikzstyle{place}=
  [circle,thick,draw=blue!75,fill=blue!20,minimum size=6mm]

\tikzstyle{transition}=
  [rectangle,thick,draw=black!75,fill=black!20,minimum size=4mm]

  \usetikzlibrary{
	cd,
	math,
	decorations.markings,
	positioning,
	arrows.meta,
	shapes,
	calc,
	fit,
	quotes,
	automata,
	petri,
	scopes
}

\tikzset{
	pics/netA/.style args={#1/#2/#3/#4/#5/#6}{code={

		\node [place,label=above right:$X$,colored tokens={%
			#1
		}] (-p_x) {};

		\node [transition,label=above:$\tau$, label=left:#4] (-t_tau) [above = of -p_x] {};
		
		\node [place,label=above:$Y$,colored tokens={%
			#2
		}] (-p_y) [right=of -t_tau] {};
		
		\node [transition,label=below:$\mu$, label=right:#5] (-t_mu) [below = of -p_y] {};
		
		\node [transition,label=below:$\nu$, label=left:#6] (-t_nu) [below =of -p_x] {};

		\node [place,label=below:$Z$,colored tokens={%
			#3
		}] (-p_z) [right=of -t_nu] {};
		
        \draw[->] (-p_x) -- (-t_tau);
		\draw[->] (-t_tau) -- (-p_y);
		\draw[->] (-p_y) -- (-t_mu);
		\draw[->] (-t_mu) -- (-p_x);
		\draw[->] (-p_x) -- (-t_nu);
		\draw[->] (-t_nu) -- (-p_z);

	}}
}

\tikzset{
	pics/netB/.style args={#1/#2/#3}{code={

		\node [place,label=above:$X$,colored tokens={%
			#1
		}] (-p_x) {};

		\node [transition,label=above:$\tau$, label=below:#3] (-t_tau) [right = of -p_x] {};
		
		\node [place,label=above:$Y$,colored tokens={%
			#2
		}] (-p_y) [right=of -t_tau] {};

        \draw[->] (-p_x) -- (-t_tau);
		\draw[->] (-t_tau) -- (-p_y);
	}}
}

\title{Executions in (Semi-)Integer Petri Nets \\ are Compact Closed Categories}
\author{
	Fabrizio Genovese \\
	Statebox Team\\
	\texttt{fabrizio@statebox.io}\\
	University of Oxford \\
	\texttt{fabrizio.genovese@cs.ox.ac.uk}
	\and
	Jelle Herold\\
	Statebox Team\\
	\texttt{jelle@statebox.io}
}

\def\titlerunning{Executions in (Semi-)Integer Petri Nets are Compact Closed Categories}
\def\authorrunning{F. Genovese \& J. Herold}

\begin{document}
\maketitle
\begin{abstract}
	In this work, we analyse Petri nets where places are allowed to have a negative number of tokens. For each net we build its correspondent category of executions, which is compact closed, and prove that this procedure is functorial. We moreover exhibit a procedure to recover the original net from its category of executions, show that it is again functorial, and that this gives rise to an adjoint pair. Finally, we use compact closeness to infer that allowing negative tokens in a Petri net makes the causal relations between transition firings non-trivial, and we use this to model interesting phenomena in economics and computer science.
\end{abstract}
\section{Introduction}\label{sec:introduction}
Petri nets are a well known tool to study concurrent systems, and have been around for decades~\cite{Petri2008}. Intuitively, a Petri net consists of a set of \emph{places}, pictorially depicted as circles, and a set of \emph{transitions}, represented as grey squares, that are connected to places via directed edges, decorated with natural numbers (see Figure~\ref{fig:petri}). To avoid clutter, we omit the decoration when it is equal to 1. Places can contain \emph{tokens}, that are represented as black dots. An assignment of tokens for a given net is called a \emph{state} (Figure~\ref{fig:state}).
\begin{figure}[h!]
	\centering
	\begin{subfigure}[b]{0.32\textwidth}\centering
		\begin{tikzpicture}[node distance=1cm,>=stealth',bend angle=45,auto]
			\begin{scope}
			\node (1a) {};
			\node [transition] (1b) [below of=1a] {};
			
			\node [place,tokens=0] (2a) [right of=1a]  {}
			edge [pre] node[swap] {2} (1b);
			\node [place,tokens=0] (2b) [below of=2a]  {}
			edge [pre] (1b);
			\node [place,tokens=0] (2c)  [below of=2b] {}
			edge [pre] node {3} (1b);
			
			\node [transition] (3a) [right of =2b]     {}
			edge [pre] (2a)
			edge [pre] node[xshift=-1mm, swap] {2} (2b)
			edge [pre] (2c);
			
			\node [place,tokens=0] (4a) [above right of=3a]  {}
			edge [pre] (3a);			
			\node [place,tokens=0] (4b) [below right of=3a]  {}
			edge [pre] (3a);
			\node [transition] (5a) [below right of =4a]  {}
			edge [pre] (4a)
			edge [pre] node {4} (4b);
			\node [transition] (5b) [right of =4a]  {}
			edge [pre] node[swap] {4} (4a);
			\end{scope}
		\end{tikzpicture}
		\caption{A Petri net.}
		\label{fig:petri}
	\end{subfigure}
	\hfill 
	\begin{subfigure}[b]{0.32\textwidth}\centering
		\begin{tikzpicture}[node distance=1cm,>=stealth',bend angle=45,auto]
			\begin{scope}
			\node (1a) {};
			\node [transition] (1b) [below of=1a] {};
			
			\node [place,tokens=3] (2a) [right of=1a]  {}
			edge [pre] node[swap] {2} (1b);
			\node [place,tokens=2] (2b) [below of=2a]  {}
			edge [pre] (1b);
			\node [place,tokens=4] (2c)  [below of=2b] {}
			edge [pre] node {3} (1b);
			
			\node [transition] (3a) [right of =2b]     {}
			edge [pre] (2a)
			edge [pre] node[xshift=-1mm, swap] {2} (2b)
			edge [pre] (2c);
			
			\node [place,tokens=1] (4a) [above right of=3a]  {}
			edge [pre] (3a);			
			\node [place,tokens=2] (4b) [below right of=3a]  {}
			edge [pre] (3a);
			\node [transition] (5a) [below right of =4a]  {}
			edge [pre] (4a)
			edge [pre] node {4} (4b);
			\node [transition] (5b) [right of =4a]  {}
			edge [pre] node[swap] {4} (4a);
			\end{scope}
		\end{tikzpicture}
		\caption{A state, before firing.}
		\label{fig:state}
	\end{subfigure}
	\hfill
	\begin{subfigure}[b]{0.32\textwidth}\centering
		\begin{tikzpicture}[node distance=1cm,>=stealth',bend angle=45,auto]
			\begin{scope}
			\node (1a) {};
			\node [transition] (1b) [below of=1a] {};
			
			\node [place,tokens=2] (2a) [right of=1a]  {}
			edge [pre] node[swap] {2} (1b);
			\node [place,tokens=0] (2b) [below of=2a]  {}
			edge [pre] (1b);
			\node [place,tokens=3] (2c)  [below of=2b] {}
			edge [pre] node {3} (1b);
			
			\node [transition, label=above:$\blacktriangledown$] (3a) [right of =2b]     {}
			edge [pre] (2a)
			edge [pre] node[xshift=-1mm, swap] {2} (2b)
			edge [pre] (2c);
			
			\node [place,tokens=2] (4a) [above right of=3a]  {}
			edge [pre] (3a);			
			\node [place,tokens=3] (4b) [below right of=3a]  {}
			edge [pre] (3a);
			\node [transition] (5a) [below right of =4a]  {}
			edge [pre] (4a)
			edge [pre] node {4} (4b);
			\node [transition] (5b) [right of =4a]  {}
			edge [pre] node[swap] {4} (4a);
			\end{scope}
		\end{tikzpicture}
		\caption{Marked transition has fired.}
		\label{fig:firing}
	\end{subfigure}
		\caption{}
\end{figure}
\noindent
We interpret  places as types, tokens as resources of the type corresponding to the place they are in, and transitions as processes that convert resources into other resources. In more detail, when a transition converts resources into other resources we say that it \emph{fires}. When firing, a transition consumes tokens in the places connected to it via an inbound edge, and produces tokens in places connected to it via an outbound edge. The number of tokens consumed/produced for each place is specified by the weighting on the edges. Firing is denoted with the symbol $\blacktriangledown$ (see Figures~\ref{fig:state} and~\ref{fig:firing}).

Petri nets were originally invented to study chemical reactions~\cite{Petri2008}, but quickly found a place in computer science as models for concurrency~\cite{Nielsen1991, Riemann1999}. This is motivated by the idea that transitions sharing some input places have to \emph{compete} for tokens to fire, and can then be thought of as \emph{concurrent processes}.

\noindent
Petri nets, as presented above, are good models for concurrency but difficult to implement: Tokens represent resources, but the model itself offers no way of ``tracking'' their history (i.e. all the proccesses by which a given token was consumed/produced), which is fundamental to turning a Petri net into actual code.
We have then to distinguish a Petri net, representing a process in the abstract, from its \emph{executions} (sometimes also called \emph{computations}), representing all the possible ways to run it.

In~\cite{Meseguer1990, Sassone1996, Sassone2000, Bruni2001} Petri nets have been characterized categorically. An outcome of this line of work has been linking Petri nets to their executions in terms of functorial relationships between categories~\cite{Sassone1995}. In this work we carry on along the same lines, as follows: In Section~\ref{sec: first definitions} we generalize the notion of Petri net allowing for negative tokens, and explain why this is desirable; in Section~\ref{sec:executions} we reshape the categorical characterization of net executions, such that the functorial relationship is preserved; in Section~\ref{sec:causality}, we will exploit some categorical properties of our model (mainly compact closedness) to show how our generalized nets have non-trivial behavior, and will provide examples of why this is useful; in Section~\ref{sec:future work} we will sketch future directions of research. 

Interestingly, our approach to Petri nets will have striking similarities with the one used in categorical quantum mechanics, both from a structural point of view -- our categories will be compact closed, and we will make great use of string diagrams throughout the paper -- and from a conceptual one -- causality flow in (semi-)integer nets will be non-trivial and similar in flavor to quantum teleportation.
\section{Integer and semi-integer Petri nets}\label{sec: first definitions}
There are many definitions of Petri net, not always equivalent. We follow an approach similar to the one used in~\cite{Sassone1995}. In the remainder of this work, we will adopt the following notations: Categorical composition $A \xrightarrow{f} B \xrightarrow{g} C$ will be denoted with $f;g$. Given a set $S$, $\Mset{S}$ will denote the set of all \emph{finite multisets} on $S$, that is, the set of all functions $S \to \naturals$ that are non-zero only on a finite subset of $S$. Similarly, $\Iset{S}$ will denote the set of all \emph{finite signed multisets} on $S$, viz. the set of all functions $S \to \integers$ that are non-zero only on a finite subset of $S$. It is worth recalling the well-known fact that for each $S$, $\Mset{S}$ is the free commutative monoid generated by $S$ under the operation of multiset union~\cite{Sassone1995}, while $\Iset{S}$ is the free abelian group generated by $S$ under multiset union and subtraction~\cite[45-46]{Joshi2003}.

\begin{definition}\label{def:Petri Category}
	A \emph{Petri net} $N$ is a 4-tuple $\Net{N}$, where:
	\begin{itemize}
		\item $P_N$ is a set, called the \emph{set of places};
		\item $T_N$ is a set, called the \emph{set of transitions};
		\item $\Pin{-}_N, \Pout{-}_N$ are functions $T_N \to \Mset{P_N}$, called \emph{input} and \emph{output}, respectively.
	\end{itemize}
	A \emph{state for the net $N$} is an element of $\Mset{P_N}$, representing how many tokens are in each place.
	
	Given nets $N := \Net{N}$ and $M := \Net{M}$, a \emph{morphism from $N$ to $M$} is a pair $\langle f,g \rangle$ where $f$ is a function $T_N \to T_M$, $g$ is a monoid homomorphism $\Mset{P_N} \to \Mset{P_M}$, and the following conditions hold:
	\begin{equation*}
	f;\Pin{-}_M = \Pin{-}_N;g \qquad f;\Pout{-}_M = \Pout{-}_N;g
	\end{equation*} 
	It is straightforward to check that Petri nets and Petri net morphisms form a category, called $\PetriN$.
\end{definition}
\noindent
A morphism of nets $N \to M$ expresses the fact that $N$ \emph{can be simulated by} $M$, as thoroughly explained in~\cite{Meseguer1990}. We are now ready to define the main objects of our investigation, generalizing the previous definition:
\begin{definition}\label{def:PetriZ and ZZ Category}
	A \emph{semi-integer Petri net} is a Petri net where states are elements of $\Iset{P_N}$. An \emph{integer Petri net} is a semi-integer Petri net where $\Pin{-}_N, \Pout{-}_N$ are functions $T_N \to \Iset{P_N}$. A morphism of semi-integer nets is defined exactly as in~\ref{def:Petri Category}, while a morphism $N \to M$ of integer nets is defined taking $g$ in the pair $\langle f,g \rangle$ to be a group homomorphism $\Iset{P_N} \to \Iset{P_M}$.
	
	\noindent
	Semi-integer nets and their morphisms form again a category, and so do integer nets and their morphisms. We denote them as $\PetriZZ$ and $\PetriZ$, respectively.
	
\end{definition}
\noindent
All in all, semi-integer Petri nets are just ordinary nets where states are allowed to have negative tokens, but transitions can only produce/consume positive ones. Integer nets, instead, also allow for transitions to consume and produce negative tokens. 

Furthermore, some readers may have noticed that the category $\PetriN$ is defined exactly as in~\cite{Sassone1995}, and that $\PetriN$ and $\PetriZZ$ are the same thing. This should not surprise, since the categories just defined do not capture any information about the net states, and only account for the underlying topology. This is part of a bigger problem, precisely that the role of states for a net has always been ambiguous. States represent the ``dynamical part'' of the net (resources produced/consumed) and are often not considered to be part of it. The impact that this ambiguity has on the definition of net executions will be evident and thoroughly discussed at the end of Section~\ref{sec:executions}.

Now we recast our graphical formalism to deal with (semi-)integer nets. We represent a negative number of tokens in a place using red dots. For instance, the transition in Figure~\ref{fig:negpetri} has $-2$ tokens in its input place and $3$ tokens in its output place. Note that $\integers$ being a group, we can always ``produce'' an equal number of positive and negative tokens in each place (see Figure~\ref{fig:group}). This has dramatic consequences on the behaviour of our nets: Now each transition can fire at will ``borrowing'' tokens from a place, that is left with an equal number of tokens of the opposite sign (see Figures~\ref{fig:neg pre firing} and~\ref{fig:neg post firing}).
\begin{figure}[h!]
	\centering
	\begin{subfigure}[b]{0.2\textwidth}\centering
			\begin{tikzpicture}[node distance=1.3cm,>=stealth',bend angle=45,auto]			
			\node[place, colored tokens={red, red}] (1a) {};
			\node[transition] (2a) [right of=1a] {}
			edge [pre] (1a);
			\node[place, tokens=3] (3a) [right of=2a] {}
			edge [pre] node[swap] {-4} (2a);
		\end{tikzpicture}
		\caption{An integer Petri net.}
		\label{fig:negpetri}
	\end{subfigure}
	~\hfill 
	\begin{subfigure}[b]{0.2\textwidth}\centering
		\begin{tikzpicture}[node distance=1.3cm,>=stealth',bend angle=45,auto]		
			\node[place, tokens=1] (1a) {};
			\node[] (2a) [right of=1a] {$=$};
			\node[place, colored tokens={black, red, black}] (3a) [right of=2a] {};
		\end{tikzpicture}
		\caption{Integer tokens.}
		\label{fig:group}
	\end{subfigure}
		~\hfill 
	\begin{subfigure}[b]{0.2\textwidth}\centering
		\begin{tikzpicture}[node distance=1.3cm,>=stealth',bend angle=45,auto]
			\node[place, tokens=1] (1a) {};
			\node[transition] (2a) [right of=1a] {}
			edge [pre] node[swap] {2} (1a);
			\node[place] (3a) [right of=2a] {}
			edge [pre] (2a);
		\end{tikzpicture}
		\caption{Before firing.}
		\label{fig:neg pre firing}
	\end{subfigure}
	~\hfill 
	\begin{subfigure}[b]{0.2\textwidth}\centering
		\begin{tikzpicture}[node distance=1.3cm,>=stealth',bend angle=45,auto]			
			\node[place, colored tokens={red}] (1a) {};
			\node[transition, label=above:$\blacktriangledown$] (2a) [right of=1a] {}
			edge [pre] node[swap] {2} (1a);
			\node[place, tokens=1] (3a) [right of=2a] {}
			edge [pre] (2a);
		\end{tikzpicture}
		\caption{After firing.}
		\label{fig:neg post firing}
	\end{subfigure}
	\caption{}
\end{figure}

\noindent
Semi-integer nets can be useful to model conflict resolution in concurrent behaviour. Petri nets are, in fact, good models for concurrent computation, but do not take into account what happens when the computation is shared by multiple agents over a non-ideal network. Consider, for instance, the net in Figure~\ref{fig:conflict petri}: Transitions $t_1$ and $t_2$ have to compete for the token in $p_1$ and they cannot both fire. Now suppose that there are two users, say $U_1$ and $U_2$, that can operate on the net, deciding which transition to fire. When a user takes a decision, it is broadcast over the network to the other user, and the overall state of the net is updated. In a realistic scenario, though, broadcasting over the network takes time: User $U_1$ could decide to fire $t_1$ and user $U_2$ could decide to fire $t_2$ while the broadcast choice of $U_1$ has still to be received, putting the overall net into an illegal state (Figure~\ref{fig:conflict petri what}).
\begin{figure}[h!]
	\centering
	\begin{subfigure}[b]{0.3\textwidth}\centering
		\begin{tikzpicture}[node distance=1cm,>=stealth',bend angle=45,auto]	
		\node[place, tokens=1] (1a) [label=left:$p_1$] {};
		\node[transition] (2a) [above right of=1a, label=right:$t_1$] {}
		edge [pre] (1a);
		\node[transition] (2b) [below right of=1a, label=right:$t_2$] {}
		edge [pre] (1a);
		\end{tikzpicture}
		\caption{}
		\label{fig:conflict petri}
	\end{subfigure}
	~\hfill 
	\begin{subfigure}[b]{0.3\textwidth}\centering
		\begin{tikzpicture}[node distance=1cm,>=stealth',bend angle=45,auto]	
		\node[place, label=center:$?$] (1a) [label=left:$p_1$] {};
		\node[transition] (2a) [above right of=1a, label=right:{$t_1$, fired by $U_1$}] {}
		edge [pre] (1a);
		\node[transition] (2b) [below right of=1a, label=right:{$t_2$, fired by $U_2$}] {}
		edge [pre] (1a);
		\end{tikzpicture}
		\caption{}
		\label{fig:conflict petri what}
	\end{subfigure}
	~\hfill 
	\begin{subfigure}[b]{0.3\textwidth}\centering
		\begin{tikzpicture}[node distance=1cm,>=stealth',bend angle=45,auto]			
		\node[place, colored tokens={red}] (1a) [label=left:$p_1$] {};
		\node[transition] (2a) [above right of=1a, label=right:{$t_1$, fired by $U_1$}] {}
		edge [pre] (1a);
		\node[transition] (2b) [below right of=1a, label=right:{$t_2$, fired by $U_2$}] {}
		edge [pre] (1a);
		\end{tikzpicture}
		\caption{}
		\label{fig:conflict integer petri}
	\end{subfigure}
	\caption{}
\end{figure}

\noindent
In such a situation we need a way to re-establish \emph{consensus}, that is, decide unambiguously in which legal state the net is. There are multiple ways to do this, but our main concern here is that the usual Petri net formalism does not have a way to represent illegal states, which is fundamental to attacking the problem. With integer states we are able to easily represent such a situation using negative tokens, as in Figure~\ref{fig:conflict integer petri}.
Re-establishing consensus from an illegal state then amounts to getting back to a state where the number of tokens in each place is non-negative.

Integer nets can instead be useful to model economic phenomena: Places can be seen as actors (or accounts) and tokens as entries in these accounts. Negative tokens then represent debit while positive tokens are credits, and transitions are bookkeeping events that convert between credits and debits. The characterization of economic phenomena in terms of process theories is object of a broader research that the Statebox team is carrying on along with multiple partners, and that also involves open games~\cite{Ghani2016}, macroeconomics~\cite{Winschel2010} and open systems~\cite{Sobocinski2010}.
\section{The category of executions of (semi-)integer Petri nets}\label{sec:executions}
Now we focus on defining executions for our nets. A suitable category of executions for $\PetriN$ has already been defined in~\cite{Sassone1995}, of which the work carried out in this section is a direct generalization. Most notably, the right categories of executions for $\PetriZZ$ and $\PetriZ$ will turn out to be compact closed, while the category of executions for $\PetriN$ is not. This is striking considering that $\PetriN$ and $\PetriZZ$ are the same category, and further highlights how the same category can be thought of in completely different ways.

The plan is as follows: We want to represent places as ``basic types'', states as monoidal products of these types (so, for instance, $A \otimes B \otimes A$ means ``a state with two tokens in $A$ and one in $B$") and transitions as morphisms between states. This tells us that our category has to be monoidal. Moreover, we want to represent negative tokens and the fact that they annihilate with positive ones, and to do so we resort to duals ($A^\star$ stands for ``$-1$ tokens in $A$''), cups and caps (representing creation/annihilation of tokens of the opposite sign). Thus, we realize that a compact closed category would be a good model to represent (semi-)integer net computations.

Counter-intuitively, we cannot require the monoidal product to be commutative, because doing so would disrupt the functorial relationship between Petri nets and their executions, as proven in~\cite[Thm. 2.2]{Sassone1995}. This will force us to do quite a lot of bookkeeping, such as keeping track of permutations in a state.
\begin{definition}
	Let $S$ be a set. $S$ can be seen as a discrete category, and one can build the free strict\footnote{Following the notation given in~\cite{KellyLaplaza1980}, by strict compact closed category we mean a category that is strict as a symmetric monoidal category, and for which the isomorphisms $(A \tensor B)^\star \simeq B^\star \tensor A^\star$, $\tensorUnit^\star \simeq \tensorUnit$ and ${A^\star}^\star \simeq A$ are all identities.} compact closed category generated by $S$, as shown in~\cite{KellyLaplaza1980}. Denote this category with $\mathcal{S}$. We define the strict compact closed category $\Gr{S}$ as $\mathcal{S}$ modulo the axioms
	\begin{equation}\label{eq:axiom}
		 \epsilon_{u^{-1}} = \sigma_{u^{-1}, u};\epsilon_{u} \qquad \eta_u;\sigma_{u^{-1},u};\epsilon_{u^{-1}} = \id{\tensorUnit} \qquad \forall u.(u \in \obj{\Gr{S}})
	\end{equation}
	Where $\sigma$ denotes symmetries and $\eta, \epsilon$ units and co-units, respectively (also called \emph{cups and caps}).
\end{definition}
\noindent
Given a set $S$, we will denote with $\IStrings{S}$ the set of finite strings of elements of $S \cup S^{-1}$, where $S^{-1}$ is the set of formal expressions $\suchthat{s^{-1}}{s \in S}$. It is easy to check that objects of $\Gr{S}$ are just elements of $\IStrings{S}$, that the monoidal product (denoted with $\tensor$) is just a concatenation of strings (with unit $\tensorUnit$ being the empty string), and that the dual of a string is its inverse when $\IStrings{S}$ is seen as the underlining set of the free group generated by $S$. Morphisms of $\Gr{S}$ are obtained as finite compositions/monoidal products of identities, binary swaps, cups and caps on elements.
\begin{remark}\label{rem:freeness of gr}
	Using the freeness of $\mathcal{S}$ one can moreover check that, for every strict compact closed category $\CategoryC$ for which the axioms in~\ref{eq:axiom} hold, and for each function $f:S \to \obj{\CategoryC}$, there is a unique strict symmetric monoidal functor $\overline{f}:\Gr{S} \to \CategoryC$, carrying the cups and caps of $\Gr{S}$ to the cups and caps of $\CategoryC$,  extending $f$.
\end{remark}
\begin{definition}
	An \emph{integer Petri category} is a strict compact closed category whose monoid of objects is the monoid $\IStrings{S}$ for some set $S$ such that, for each object, $A^\star = A^{-1}$ and the axioms in~\ref{eq:axiom} hold.
\end{definition}
\begin{definition}
	Given an integer Petri category $\CategoryC$, we call an arrow $\tau$ of $\CategoryC$ \emph{structural} if 
	\begin{equation*}
		\tau = \bigtensor_{i_1 = 1}^{m_1} \alpha_{i_1}^1; \dots ; \bigtensor_{i_n = 1}^{m_n} \alpha_{i_n}^n
	\end{equation*}
	With each $\alpha_{i_j}^j$ being an identity, a symmetry, a cup or a cap.
\end{definition}
\noindent
Intuitively, structural arrows are needed for the above-mentioned bookkeeping in modelling a net execution. We want to represent the fact that transitions consume tokens produced by other transitions by composing processes, but to do this we need symmetries to reorder the way we present tokens if needed, and cups/caps to represent the creation/annihilation of tokens of the opposite sign when they are in the same place.
\begin{definition}
	Given an integer Petri category $\CategoryC$, an arrow $\tau$ of $\CategoryC$ is \emph{primitive} if:
	\begin{itemize}
		\item $\tau$ is not structural;
		\item If $\tau = \alpha;\beta$ then $\alpha$ is structural and $\beta$ is primitive, or vice-versa;
		\item If $\tau = \alpha \tensor \beta$ then $\alpha = \id{\tensorUnit}$ and $\beta$ is primitive, or vice-versa.
	\end{itemize}
\end{definition}
\noindent
We would like primitive arrows to represent ``the actual transitions of the net''. To accomplish this we think of an arrow $\tau$ in $\CategoryC$ as primitive when it's not in the image of the functor  $\Gr{S} \to \CategoryC$ obtained lifting the obvious inclusion $S \inject \obj{\CategoryC}$. This is the best way to say that ``a primitive arrow is the smallest arrow that is not structural'', since structural arrows in $\CategoryC$ can be seen as ``always coming from $\Gr{S}$''.

As we will see shortly, this is not enough to reliably identify what stands for a transition in a Petri category, and we will have to refine this idea further to make it work.
\begin{lemma}
	There is an obvious mapping $\Multiplicity: \IStrings{S} \to \Iset{S}$ that associates to each string $s \in \IStrings{S}$ an integer multiset $S \to \integers$:
	\begin{equation*}
		\Multiplicity(s)(p) := \text{ Occurrences of $p$ in $s$ } - \text{ Occurrences of $p^{-1}$ in $s$}
	\end{equation*}
\end{lemma}
\begin{restatable}{lemma}{structuralsubcategory}
	Let $\CategoryC$ be an integer Petri category and let $\IStrings{S}$ be its monoid of objects. Given an integer multiset $\nu:S \to \integers$, consider the set of objects of $\CategoryC$ such that their image through $\Multiplicity$ is $\nu$, along with structural arrows between them. This gives a subcategory of $\CategoryC$, denoted with $\Gr{\CategoryC, \nu}$.
\end{restatable}
\noindent
If $\Iset{S}$ can be seen as the free abelian group on $S$, $\IStrings{S}$ clearly stands for the free group on $S$. $\Multiplicity$ acts identifying all the objects of $\IStrings{S}$ that would end up being identified if we were to quotient it by introducing commutativity. $\Gr{\CategoryC, \nu}$ then is the category of all possible operations that we can make on an element in $\IStrings{S}$ without changing the equivalence class it is sent to, viz. all the possible bookkeeping we can do on a object without altering the net state it corresponds to.

The following definitions are a direct generalization of the ones given in~\cite{Sassone1995}.
\begin{definition}
	Let $\CategoryC$ be an integer Petri category and let $\IStrings{S}$ be its monoid of objects. Given integer multisets $\nu, \nu':S \to \integers$, a \emph{\underline{transition}}\footnote{To be unambiguous, we will distinguish transitions in the net context from \underline{transitions} in the categorical context by underlining.} of $\CategoryC$ is a natural transformation $\tau: \pi_{\CategoryC, \nu} \to \pi_{\CategoryC, \nu'}$ whose components are all primitive, where $\pi_{\CategoryC, \nu}$ and $\pi_{\CategoryC, \nu'}$ are the obvious compositions of projection and inclusion functors:
	\begin{equation*}
		\pi_{\CategoryC, \nu} : \Gr{\CategoryC, \nu'} \cartesianTensor \Gr{\CategoryC, \nu} \xrightarrow{\pi_1} \Gr{\CategoryC, \nu} \inject \mathcal{C} \qquad 	\pi_{\CategoryC, \nu'} : \Gr{\CategoryC, \nu'} \cartesianTensor \Gr{\CategoryC, \nu} \xrightarrow{\pi_2} \Gr{\CategoryC, \nu'} \inject \mathcal{C}
	\end{equation*}
	We say that a strong monoidal functor between Petri categories $F:\CategoryC \to \CategoryD$ \emph{preserves \underline{transitions}} if $F$ carries structural arrows to structural arrows and for each \underline{transition} $\tau$ of $\CategoryC$ there is a \underline{transition} $\theta$ of $\CategoryD$ such that $F\tau_{u,v} = \theta_{Fu,Fv}$, where by $\tau_{u,v}$ we denote the components of $\tau$ (similarly for $\theta$).
\end{definition}
\noindent
As one can imagine, \underline{transitions} in the previous definition represent the actual transitions of a given Petri net in its category of executions. To see this note that if we want to represent a transition as a process, then clearly we do not want to consider it as a different one if we permute the objects in its input or output. This is because in a Petri net a transition only cares about the number of tokens it consumes (produces) from (in) a place, and not about the order in which these tokens are received (sent). Similarly, we want to ``ignore'' all the pairs of type $s, s^{-1}$ showing up in the process input and output, since these couples correspond to a $0$ when translated to multisets, meaning that again in the Petri net formalism these pairs are ``not seen'' by any transition. 

Clearly in the world of categories things are different, since morphisms of a category are sensitive to object order and/or presence of object+dual pairs. Requiring a morphism to be indifferent to this amounts exactly to asking that it commutes with swaps, cups and caps. In categorical terms it means requiring that the morphism is the component of a natural transformation between functors expressing the action of commuting and introducing/removing such pairs.

Finally, we see that if transitions in a net become natural transformations in the corresponding category of executions, then we want a notion of morphism between these categories that corresponds, functorially, to the notion of morphism we have between nets. Since a morphism between nets sends transitions to transitions, the natural requirement in the category of executions is that morphisms (viz. functors) preserve the natural transformations that are \underline{transitions}.

\begin{definition}
	Given two integer Petri categories $\CategoryC, \CategoryD$ we define the relation $\mathfrak{R}_{\CategoryC, \CategoryD}$ on \underline{transition}-preserving functors $F,G:\CategoryC \to \CategoryD$ saying that $F,G \in \mathfrak{R}_{\CategoryC, \CategoryD}$ if there are natural transformations $\tau:F \to G$ and $\tau':G \to F$ such that their components are all structural arrows. For each couple $\CategoryC, \CategoryD$, $\mathfrak{R}_{\CategoryC, \CategoryD}$ is an equivalence relation. Moreover,
	\begin{equation*}
		(F,G) \in \mathfrak{R}_{\CategoryC, \CategoryD} \wedge (F',G') \in \mathfrak{R}_{\CategoryD,\mathcal{E}} \implies (F;F',G;G') \in \mathfrak{R}_{\CategoryC, \mathcal{E}}
	\end{equation*}
\end{definition}
\noindent
This definition is, again, in line with the idea that we want to characterize functors only by looking at what they do to \underline{transitions}. If they differ only in the way they handle the bookkeeping morphisms, then they should be regarded as the same. It is easy to check that identity functors preserve \underline{transitions}, as does composition of \underline{transition}-preserving functors. We moreover have:
\begin{restatable}{lemma}{frakrisacongruence}
	$\mathfrak{R}$ is a congruence. If $F: \CategoryC \to \CategoryD$ preserves \underline{transitions} and $\theta_{Fu,Fv} = F\tau_{u,v} = \theta'_{Fu,Fv}$, then $\theta' = \theta'$ .  Moreover, if $(F,G) \in \mathfrak{R}_{\CategoryC,\CategoryD}$ and $F, G$ preserve \underline{transitions}, $F\tau_{u,v} = \theta_{Fu,Fv}$ iff $G\tau_{u,v} = \theta_{Gu,Gv}$.
\end{restatable}
\noindent
This is enough to ensure that the following definition is correct:
\begin{definition}
	We define the category of generalized Petri executions, $\GExPetri$, as having integer Petri categories as objects and \underline{transition}-preserving functors modulo $\mathfrak{R}$ as morphisms. Note that objects of $\GExPetri$ are compact closed categories, but $\GExPetri$ is not compact closed itself.
\end{definition}
\noindent
Now we are finally ready to prove the main theorem of this work, that functorially associates, to each integer Petri net, a semantics representing its computations.
\begin{restatable}{theorem}{maintheorem}
	Consider a net $N \eqdef \Net{N} \in \obj{\PetriZ}$. We can associate to $N$ the strict compact closed category including $\Gr{P_N}$ as a subcategory, plus the arrows defined by the following inference rules:
	\begin{gather*}
	\frac{t \in T_N}{t_{u,v} \in \Hom{\Fold{N}}{u}{v}} \quad \forall u,v.(\Multiplicity(u)= \Pin{t} \wedge \Multiplicity(v)= \Pout{t})\\
	\\
	\frac{\alpha \in \Hom{\Fold{N}}{u}{v}, \,\, \beta \in \Hom{\Fold{N}}{u'}{v'}}{\alpha \tensor \beta \in \Hom{\Fold{N}}{u \tensor u'}{v \tensor v'}} 
	\qquad 		
	\frac{\alpha \in \Hom{\Fold{N}}{u}{v}, \,\, \beta \in \Hom{\Fold{N}}{v}{w}}{\alpha ; \beta \in \Hom{\Fold{N}}{u}{w}} 
	\end{gather*}
	in which the following family of axioms holds:
	\begin{equation}\label{eq: axiom}
	p;t_{u',v'} = t_{u,v};q \qquad  \forall p, q. (p \in \Hom{\Gr{P_N}}{u}{u'} \wedge q \in \Hom{\Gr{P_N}}{v}{v'})
	\end{equation}
	This correspondence can be extended to a functor $\Fold{-}: \PetriZ \to \GExPetri$. 
	
	Similarly, to each category $\CategoryC \in \obj{\PetriZ}$ we can associate the net $\Net{N}$, where:
	\begin{itemize}
		\item $P_N$ is the generating set of $\obj{\CategoryC}$.
		\item $T_N \eqdef \bigcup_{\nu,\nu' \in \Iset{P_N}}\suchthat{\tau \in \Nat{}{\pi_{\CategoryC, \nu}}{\pi_{\CategoryC, \nu'}}}{\text{$\tau$ is a \underline{transition}}}$
		\item $\Pin{\tau: \pi_{\CategoryC, \nu} \to \pi_{\CategoryC, \nu'}} = \nu$
		\item $\Pout{\tau: \pi_{\CategoryC, \nu} \to \pi_{\CategoryC, \nu'}} = \nu'$
	\end{itemize}
	This correspondence can again be extended to a functor $\UnFold{-}: \GExPetri \to \PetriZ$, producing an adjunction $\Fold{-} \vdash \UnFold{-}$.
\end{restatable}
\noindent
Finally, we can restrict the category $\GExPetri$ and the functor $\UnFold{-}$ to get the following:
\begin{restatable}{corollary}{firstequivalence}
	Call $\ExPetriZ$ the full subcategory of $\GExPetri$ consisting of integer Petri categories whose morphisms are generated only from compositions and monoidal products of symmetries, cups, caps and \underline{transition} components, modulo the compact closed categories axioms, the axioms in~\ref{eq:axiom} and the axioms in~\ref{eq: axiom}. Then $\PetriZ \approx \ExPetriZ$.
\end{restatable}
\noindent
We found a suitable category to represent executions of integer Petri nets. Now we want to find a subcategory of $\ExPetriZ$ that can represent executions of semi-integer nets. This is indeed easy, and amounts to putting an obvious requirement on what \underline{transitions} look like:
\begin{definition}
	We call an object $\CategoryC$ of $\ExPetriZ$ \emph{positive} if each \underline{transition} in $\CategoryC$ is of type $\tau: \pi_{\CategoryC, \nu} \to \pi_{\CategoryC, \nu'}$ with all the elements in the images of $\nu, \nu'$ being $\geq 0$. Positive objects of $\ExPetriZ$ and morphisms between them form a full subcategory, denoted with $\ExPetriZZ$.
\end{definition}
\begin{restatable}{theorem}{secondequivalence}
	There is an equivalence $\PetriZZ \approx \ExPetriZZ$, and thus an equivalence
	\begin{equation}
		\ExPetriN \approx \PetriN = \PetriZZ \approx \ExPetriZZ
	\end{equation}
	Where $\ExPetriN$ and $\PetriN$ are the categories denoted as \underline{\textbf{PSSMC}} and \underline{\textbf{Petri}}, respectively, in~\cite{Sassone1995}.
\end{restatable}
\noindent
This result says that the category of executions for Petri nets defined in~\cite{Sassone1995}, whose objects are monoidal but not compact closed categories, is equivalent to $\ExPetriZZ$, that is a category of compact closed categories. This should not surprise: $\PetriN$ and $\PetriZZ$ are the same category, but represent different things: The fact is that they are different in the definition of what a state is, which is not accounted for in their categorical structure. Different definitions of state are then embedded in the structure of the objects of $\ExPetriN$ and $\ExPetriZZ$, while the morphism structure of $\ExPetriN$ and $\ExPetriZZ$ models how different nets interact with each other. This result then says that semi-integer nets simulate and interact with each other exactly as normal nets do: Different types of objects, representing different types of executions, are connected to each other in the same way, that only depends on the underlying topology.
\section{The causal structure of net computations}\label{sec:causality}

Now we want to use the properties of compact closed categories to gain a better insight into how the executions of (semi-)integer Petri nets work, that is, we want to look inside the objects of $\ExPetriZ$ and $\ExPetriZZ$. Compact closed categories admit a well-known graphical calculus~\cite{Selinger2010, CoeckeKissinger2016}, that we will extensively use to represent what occurs. This calculus has been extensively used in the study of categorical quantum mechanics and the reader familiar with this topic should consider carefully such conceptual links while reading the following. First of all, a quick recap: In the graphical calculus for compact closed categories we express categorical facts as diagrams, that are read left to right. Objects are drawn as wires and arrows as boxes. The wire standing for the monoidal unit $I$ is not drawn, or it is drawn as dashed when its presence needs to be emphasized (Figure~\ref{fig:wires}). Composition of arrows $f;g$ is just wiring the outputs of $f$ into the inputs of $g$, while monoidal products are depicted putting boxes and wires next to each other (Figure~\ref{fig:product}). Swaps, cups and caps are represented as in Figure~\ref{fig:cupcapswap}. 
\begin{figure}[h!]
	\centering
	\begin{subfigure}[c]{0.2\textwidth}\centering
		\begin{tikzpicture}[oriented WD]
		\node[bb={1}{1}] (B) {$f$};
		\node[transparent, bb port length=8pt, bb={3}{3},bbx=0.8cm, bby=1cm, fit={(B)}] (X) {};
		\draw[label] 
		node [left=2pt of X_in1] {$u$}
		node [right=2pt of X_out1] {$u$}
		node [left=2pt of X_in2] {$u$}
		node [right=2pt of X_out2] {$v$}
		node [left=2pt of X_in3] {$I$}
		node [right=2pt of X_out3] {$I$}
		;
		\draw (X_in1') to (X_out1');
		\draw[dashed] (X_in3') to (X_out3');
		\draw (X_in2') to (B_in1);
		\draw (B_out1) to (X_out2');
		\end{tikzpicture}
		\caption{Object, morphism, monoidal unit.}
		\label{fig:wires}
	\end{subfigure}
	~\hfill 
	\begin{subfigure}[c]{0.2\textwidth}\centering
		\begin{tikzpicture}[oriented WD, bb small, bb port sep = 3]
		\node[bb={1}{1}, bby=0.2cm] (h) {$h$};
		\node[above left = 0cm and 0.2cm of h, bb={1}{1}] (f) {$f$};
		\node[above right = 0cm and 0.2cm of h, bb={1}{1}] (g) {$g$};
		\draw[label] 	
		node [above left=2pt and 4pt of g_in1] {$u''$}
		;
		\node[transparent, bb port length=8pt, bb={2}{2}, bbx=0.6cm, fit={($(f.north) + (0,2)$) ($(h.south) - (0,2)$) (f.west) (g.east)}] (X) {};
		\draw[label] 		
		node [left=2pt of X_in1] {$u$}
		node [right=2pt of X_out1] {$u'$}
		node [left=2pt of X_in2] {$v$}
		node [right=2pt of X_out2] {$v'$}
		;
		\draw (X_in1') to (f_in1);
		\draw (X_in2') to (h_in1);
		\draw (g_out1) to (X_out1');
		\draw (h_out1) to (X_out2');
		\draw (f_out1) to (g_in1');
		\end{tikzpicture}
		\caption{$u \otimes v \xrightarrow{(f;g)\otimes h} u' \otimes v'$}
		\label{fig:product}
	\end{subfigure}
	~\hfill 
	\begin{subfigure}[c]{0.2\textwidth}\centering
		\begin{tikzpicture}[oriented WD, bb small, bby =.6cm, bb port length=8pt]
		
		\node[transparent, bb min width=1cm, bb={4}{4}] (X) {};
		\draw[label] 
		node [left=2pt of X_in1] {$u$}
		node [right=2pt of X_out1] {$v$}
		node [left=2pt of X_in2] {$v$}
		node [right=2pt of X_out2] {$u$}
		node [left=2pt of X_in3] {$u$}
		node [right=2pt of X_out3] {\small{$u^{-1}$}}
		node [left=2pt of X_in4] {\small{$u^{-1}$}}
		node [right=2pt of X_out4] {$u$}
		;

		\draw (X_in1') to (X_out2');
		\draw (X_in2') to (X_out1');
		\draw (X_out3') to [out=180] (X_out4');
		\draw (X_in3') to [in=0] (X_in4');
		\end{tikzpicture}
		\caption{Swap (top), cap (b. left), cup (b. right).}
		\label{fig:cupcapswap}
	\end{subfigure}
	~\hfill 
	\begin{subfigure}[c]{0.2\textwidth}\centering
		\begin{tikzpicture}[oriented WD, bb small, bby =.4cm, bb port length=2pt]		
		\node[transparent, bb min width=0.3cm, bb={6}{6}] (X) {};
		\draw[label] 
		node [right=2pt of X_out5] {\small{$u^{-1}$}}
		node [right=2pt of X_out6] {$u$}
		node [left=2pt of X_in1] {\small{$u^{-1}$}}
		node [left=2pt of X_in2] {$u$}
		;	
		
		\draw (X_in1') to (X_out2');
		\draw (X_in2') to (X_out1');		
		\draw (X_out1) to [in=0] (X_out2);
		
		\draw (X_in3) to [out=180] (X_in4);
		\draw (X_in3') to (X_out4');	
		\draw (X_in4') to (X_out3');		
		\draw (X_out3) to [in=0] (X_out4);
		
		\draw (X_in5') to (X_out6');
		\draw (X_in6') to (X_out5');
		\draw (X_in5) to [out=180] (X_in6);

		\node[right =0.5cm of X] (equals) {\small{$=$}};
		\node[above =0.4cm of equals] (equals1) {\small{$=$}};
		\node[below =0.4cm of equals] (equals2) {\small{$=$}};		
		
		\node[transparent, bb min width=0.2cm,  right = 0.5cm of equals, bb={6}{6}] (Y) {};
		\draw[label] 
		node [right=2pt of Y_out5] {\small{$u^{-1}$}}
		node [right=2pt of Y_out6] {$u$}
		node [left=2pt of Y_in1] {\small{$u^{-1}$}}
		node [left=2pt of Y_in2] {$u$}
		;	
		
		\draw (Y_in1') to (Y_out1');
		\draw (Y_in2') to (Y_out2');	
		\draw (Y_out1) to [in=0] (Y_out2);
		
		\draw[dashed] (Y_in3') to (Y_out3');	
		\draw[dashed] (Y_in4') to (Y_out4');				
		\draw[dashed] (Y_in3) to [out=270, in=90] (Y_in4);
		\draw[dashed] (Y_out3) to [out=270, in=90] (Y_out4);		
		
		\draw (Y_in5') to (Y_out5');
		\draw (Y_in6') to (Y_out6');
		\draw (Y_in5) to [out=180] (Y_in6);
		\end{tikzpicture}
		\caption{Axioms~\ref{eq:axiom} and a consequence (bottom).}
		\label{fig:Petri axioms}
	\end{subfigure}
	\caption{Graphical calculus for compact closed categories.}
\end{figure}

\noindent
In a integer Petri category, moreover, the axioms in~\ref{eq:axiom} hold, and can be represented graphically as in Figure~\ref{fig:Petri axioms} (note that the third line in Figure~\ref{fig:Petri axioms} is a consequence of the first, and has been added explicitly to give a sense of symmetry). The axioms, now that they are shown graphically, have a clear interpretation: The first and the third line in Figure~\ref{fig:Petri axioms} are interpreted as ``it does not matter how you deform and twist them, in the end they are still a cap and a cup''. The axiom in the second line represents the fact that we are not interested in registering events when ``nothing happens'': The symbol on the left represents the creation of a pair of type $u \otimes u^{-1}$ followed by its annihilation, while the one on the right is the monoidal unit. The axiom just says that if no element of the created couple undergoes any sort of process/transformation before the annihilation, then we may as well forget that the event happened.
\begin{figure}[h]
	\centering
	\begin{tikzpicture}[oriented WD, bbx = 1cm, bby =.3cm, bb min width=1cm, bb port sep=1]
	\node[bb={6}{6}] (B) {$\tau$};
	
	\node[transparent, bb port length=8pt, bb={6}{6}, fit={($(B.north)+(0,-1)$) (B.east) (B.west) ($(B.south)-(0,-1)$)}] (X) {};
	\draw[label] 
	node [left=8pt of X_in1] {$u_1$}
	node [right=8pt of X_out1] {$v_1$}
	node [left=8pt of X_in2] {$u_2$}
	node [right=8pt of X_out2] {$v_2$}
	node [left=8pt of X_in3] {$u_3$}
	node [right=8pt of X_out3] {$v_3$}
	node [left=8pt of X_in4] {$u_4$}
	node [right=8pt of X_out4] {$v_4$}
	node [left=8pt of X_in5] {$u_4^{-1}$}
	node [right=8pt of X_out5] {$v_4^{-1}$}
	node [left=8pt of X_in6] {$u_5$}
	node [right=8pt of X_out6] {$v_5$}
	;
	
	\draw (X_in1') to (B_in1);
	\draw (X_in2') to (B_in3);
	\draw (X_in3') to (B_in2);
	\draw (X_in4') to (B_in4);
	\draw (X_in5') to (B_in5);
	\draw (X_in6') to (B_in6);
	\draw (X_in4) to [out=180] (X_in5);
	\draw (X_out4) to [in=0] (X_out5);
	
	\draw (B_out1) to (X_out1');
	\draw (B_out2) to (X_out3');
	\draw (B_out3) to (X_out2');
	\draw (B_out4) to (X_out4');
	\draw (B_out5) to (X_out5');
	\draw (B_out6) to (X_out6');
	\draw (X_in4) to [out=180] (X_in5);
	\draw (X_out4) to [in=0] (X_out5);
	
	\node[right =2.5cm of B] (equals) {\huge{$=$}};
	
	\node[bb={4}{4}, right =2.5cm of equals] (B1) {$\tau$};
	
	\node[transparent, bb port length=8pt, bb={4}{4}, fit={($(B1.north)+(0,-1)$) (B1.east) (B1.west) ($(B1.south)-(0,-1)$)}] (X1) {};
	\draw[label] 
	node [left=8pt of X1_in1] {$u_1$}
	node [right=8pt of X1_out1] {$v_1$}
	node [left=8pt of X1_in2] {$u_2$}
	node [right=8pt of X1_out2] {$v_2$}
	node [left=8pt of X1_in3] {$u_3$}
	node [right=8pt of X1_out3] {$v_3$}
	node [left=8pt of X1_in4] {$u_5$}
	node [right=8pt of X1_out4] {$v_5$}
	;
	
	\draw (X1_in1') to (B1_in1);
	\draw (X1_in2') to (B1_in2);
	\draw (X1_in3') to (B1_in3);
	\draw (X1_in4') to (B1_in4);
	
	\draw (B1_out1) to (X1_out1');
	\draw (B1_out2) to (X1_out2');
	\draw (B1_out3) to (X1_out3');
	\draw (B1_out4) to (X1_out4');
	\end{tikzpicture}
	\caption{Transitions can be depicted as boxes with the property above, for suitable $u_i$ and $v_i$.}
	\label{fig:transitions}
\end{figure}

\noindent
The first thing that we want to do is to see what a \underline{transition} looks like graphically. Given a category in $\ExPetriZ$, we know that the only generating morphisms are \underline{transition} components. The morphisms $\tau_{u,v}$ are all ``avatars'' of the same \underline{transition} $\tau$, to be used in different contexts depending on the bookkeeping we have to do. 
Since these components all stand for the same transition in the net $\UnFold{N}$, the only data that matters is the \underline{transition} they are part of, and in diagrams we can safely omit the component indexes, as in Figure~\ref{fig:transitions}:
Here the two boxes are clearly representing different components of the \underline{transition} $\tau$, and the equality, holding for any suitable choice of domain/co-domain (viz. objects that differ only by structural arrows, and get sent to the same multiset) is the exact graphical embodiment of $\tau$ being a natural transformation commuting with structural arrows.
\begin{figure}[h]
	\centering
	\begin{subfigure}[b]{0.48\textwidth}\centering
		\resizebox{\textwidth}{!}{
			\begin{tikzpicture}
			\pgfmathsetmacro\bS{4.5}
			\pgfmathsetmacro\hkX{0.71}
			\pgfmathsetmacro\kY{1.5}
			\pgfmathsetmacro\hkY{\kY*0.5}
			
			\draw pic (m0) at (0,1) {netB={{}/{}/{}}};
			\draw pic (m0) at (\bS,1) {netB={{black, red}/{}/{}}};
			\draw pic (m1) at ({2 * \bS},1) {netB={{red}/{black}/{$\blacktriangle$}}};
			
			\begin{scope}[very thin, shift={(-1.35, 0)}]
			\foreach \j in {1,...,2} {
				\pgfmathsetmacro \k { \j * \bS + \hkX};
				\draw[gray,dashed] (\k,0) -- (\k,-2.3);
				\draw[gray] (\k,2) -- (\k,0);
			}
			\end{scope}
			
			\begin{scope}[
			shift={(-0.65,-1)}, xscale=\bS, yscale={-1}, oriented WD,
			bbx = 1cm, bby =.4cm, bb min width=1cm, bb port sep=1]
			1
			
			\draw node [fill=white,bb={1}{1}] (Tau) at (2,0) {$\tau$};
			
			\draw (Tau_in1) -- node[above] {$X$} (1,0)
			to[out=-180,in=-90] (0.75,0.5) to[out=90,in=180]
			(1,1) -- node[xshift=-4mm, above] {$X^{-1}$} (2,1)
			-- node[xshift=7.5mm, above] {$X^{-1}$} (3,1);
			
			\draw (Tau_out1) -- node[above] {$Y$} (3,0);
			
			\end{scope}
			\end{tikzpicture}
		}
		\caption{Turning a legal state into an illegal one.}
		\label{fig:turning illegal}
	\end{subfigure}
	\hfill
	\begin{subfigure}[b]{0.48\textwidth}\centering
		\resizebox{\textwidth}{!}{
			\begin{tikzpicture}
			\pgfmathsetmacro\bS{4.5}
			\pgfmathsetmacro\hkX{0.71}
			\pgfmathsetmacro\kY{1.5}
			\pgfmathsetmacro\hkY{\kY*0.5}
			
			\draw pic (m0) at (0,1) {netB={{black}/{red}/{}}};
			\draw pic (m0) at (\bS,1) {netB={{}/{black, red}/{$\blacktriangle$}}};
			\draw pic (m1) at ({2 * \bS},1) {netB={{}/{}/{}}};
			
			\begin{scope}[very thin, shift={(-1.35, 0)}]
			\foreach \j in {1,...,2} {
				\pgfmathsetmacro \k { \j * \bS + \hkX};
				\draw[gray,dashed] (\k,0) -- (\k,-2.3);
				\draw[gray] (\k,2) -- (\k,0);
			}
			\end{scope}
			
			\begin{scope}[
			shift={(-0.65,-1)}, xscale=\bS, yscale={-1}, oriented WD,
			bbx = 1cm, bby =.4cm, bb min width=1cm, bb port sep=1]
			
			
			\draw node [fill=white,bb={1}{1}] (Tau) at (1,0) {$\tau$};
			
			\draw (Tau_in1) -- node[above] {$X$} (0,0);
			
			\draw (Tau_out1) -- node[above] {$Y$} (2,0)
			to[out=0,in=-90] (2.25,0.5) to[out=90,in=0]
			(2,1) -- node[xshift=7.5mm, above] {$Y^{-1}$}
			(1,1) -- node[xshift=-4mm, above] {$Y^{-1}$} (0,1);
			
			\end{scope}
			\end{tikzpicture}
		}
		\caption{Turning an illegal state into a legal one.}
		\label{fig:turning legal}
	\end{subfigure}
\end{figure}

\noindent
Using cups and caps, we can see how we are now able to represent executions of Petri nets that weren't representable before. Look, for instance, at  Figure~\ref{fig:turning illegal}: Above, we are representing what happens from the point of view of Petri nets, while below we depict how the execution is built as the transition fires and couples of positive/negative tokens are produced/deleted. Vertical dashed lines separate consecutive instants in time. In the context of semi-integer nets, the net in Figure~[\ref{fig:turning illegal}] depicts an execution that turns a legal state into an illegal one. This represents well the situation detailed in Section~\ref{sec: first definitions}, where we deduced that a net could always fire ``borrowing'' some tokens from a place. In the context of integer nets, the same figure may represent the idea of moving money from a given account (the first place) to another (the second place) without having it. Negative tokens then represent the necessary debt one has to make in order to take that money out of the account.

On the contrary, again in the context of semi-integer nets, the net in Figure~\ref{fig:turning legal} turns an illegal state into a legal one, literally ``deleting'' a negative token. From the point of view of integer Petri nets, this can be interpreted as the act of extinguishing a debt.
\begin{figure}[h]\centering
	\resizebox{0.7\textwidth}{!}{
		\begin{tikzpicture}
		\pgfmathsetmacro\bS{4.5}
		\pgfmathsetmacro\hkX{0.71}
		\pgfmathsetmacro\kY{1.5}
		\pgfmathsetmacro\hkY{\kY*0.5}
		
		\draw pic (m0) at (0,1) {netB={{}/{red}/{}}};
		\draw pic (m0) at (\bS,1) {netB={{black, red}/{red}/{}}};
		\draw pic (m1) at ({2 * \bS},1) {netB={{red}/{red, black}/{$\blacktriangle$}}};
		\draw pic (m0) at ({3 * \bS},1) {netB={{red}/{}/{}}};
		
		\begin{scope}[very thin, shift={(-1.35, 0)}]
		\foreach \j in {1,...,3} {
			\pgfmathsetmacro \k { \j * \bS + \hkX};
			\draw[gray,dashed] (\k,0) -- (\k,-4);
			\draw[gray] (\k,2) -- (\k,0);
		}
		\end{scope}
		
		\begin{scope}[
		shift={(-0.65,-1)}, xscale=\bS, yscale={-1}, oriented WD,
		bbx = 1cm, bby =.4cm, bb min width=1cm, bb port sep=1]
		
		\draw (0,0) -- node[above] {$Y^{-1}$} (1,0)
		-- node[above] {$Y^{-1}$} (2,0)
		-- node[above] {$Y^{-1}$} (3,0)
		to[out=0,in=-90] (3.25, 0.5)
		to[out=90,in=0] (3.0, 1);
		
		\draw node [fill=white,bb={1}{1}] (Tau) at (2,1) {$\tau$};
		
		\draw (Tau_out1) -- node[xshift=-7.5mm, above] {$Y$} (3,1);
		
		\draw (Tau_in1) -- node[xshift=4mm, above] {$X$} (1,1)
		to[out=-180,in=-90] (0.75,1.5) to[out=90,in=180]
		(1,2) -- node[above] {$X^{-1}$}
		(2,2) -- node[above] {$X^{-1}$}
		(3,2) -- node[above] {$X^{-1}$} (4,2);
		
		\end{scope}
		\end{tikzpicture}
	}
	\caption{The transpose of a transition acts in the opposite direction on tokens of opposite sign.}
	\label{fig:transpose}
\end{figure}

\noindent
Combining the concepts considered above, we can take the \emph{transpose} of any transition, as shown in Figure~\ref{fig:transpose}. This is demonstrates a duality that is ubiquitous when working with compact closed categories: Each transition can naturally be seen as a ``forward-acting'' process or as a ``backward-acting'' one on the dual entities. If a transition in a semi-integer net is thought of as ``taking positive tokens from some place $X$ to some place $Y$'', Figure~\ref{fig:transpose} shows how the same transition can be thought of as ``taking negative tokens from $Y$ to $X$''. This means that if we are in an illegal state represented by a negative token in some place, we can shift that token backwards by firing transitions pointing at that place. This makes indeed sense: Firing a transition $t$ from $X$ to $Y$, with a negative token in $Y$, amounts to saying ``If I have any way to produce a positive token in $X$, then I can automatically fix my problem in $Y$ by firing $t$. So I might as well say that $t$ shifts my problem in $Y$ to a problem in $X$''.

\noindent
In integer Petri nets, the transpose can be seen as switching from thinking in terms of credit to thinking in terms of debit: Paying someone could either mean ``giving money'' or ``acquiring debt''.
\begin{wrapfigure}{r}{0.5\textwidth}\centering
	\vspace{4pt}
	\begin{tikzpicture}[node distance=1.3cm,>=stealth',bend angle=45,auto]		
	\node[place] (1a) {};
	\node[transition] (2a) [right of=1a] {}
	edge [pre] node[swap] {1} (1a);
	\node[place] (3a) [right of=2a] {}
	edge [pre] node[swap] {-1} (2a);
	\end{tikzpicture}
	\caption{The Inversion transition}
	\vspace{-2pt}
	\label{fig:hadamard net}
\end{wrapfigure}
In integer nets we can also have mixed transitions, like the \emph{inversion transition} depicted in Figure~\ref{fig:hadamard net}. We think of it as a process consuming money and producing debt. The transpose, in this case, is again a process that consumes money and produces debt, but flowing in the other direction.

Now we can use this machinery to tackle the conflict resolution problem exposed in Section~\ref{sec: first definitions}. The chain of events is represented in Figure~\ref{fig:resolution}:
First, a user $U_1$ fires $\tau$, consuming the only token present in $X$. Then user $U_2$, unaware of $U_1$'s action, fires $\nu$, putting the net into an illegal state. Finally, $U_1$ puts the net back to a legal state, ``giving back'' to $X$ the resource used before, by firing $\mu$. We clearly see how cups and caps \emph{already give us the solution for our conflict problem}: If we apply the yanking equations and straighten the line, \emph{we get a sequence of legal states}, namely the sequence of firings ``$\tau$, then $\mu$, then $\nu$''. We witness how the graphical formalism for executions makes the solution to our problem easy to understand: If the vertical bars separate instants in time as it flows in the real word, the wire in the string diagram represents the flow of time \emph{according to the net itself}: Eliminating negative tokens previously produced amounts to reshuffling the order of events in the net.
\begin{figure}[h]
	\resizebox{\textwidth}{!}{
		\begin{tikzpicture}
		\pgfmathsetmacro\bS{3.5}
		\pgfmathsetmacro\hkX{(\bS/3.5)}
		\pgfmathsetmacro\kY{-1.5}
		\pgfmathsetmacro\hkY{\kY*0.5}
		
		\draw pic (m0) at (0,0) {netA={{black}/{}/{}/{}/{}/{}}};
		\draw pic (m1) at (\bS,0) {netA={{}/{black}/{}/{$\blacktriangleright$}/{}/{}}};
		\draw pic (m2) at ({2 * \bS},0) {netA={{black,red}/{black}/{}/{}/{}/{}}};
		\draw pic (m3) at ({3 * \bS},0) {netA={{red}/{black}/{black}/{}/{}/{$\blacktriangleright$}}};
		\draw pic (m4) at ({4 * \bS},0) {netA={{black, red}/{}/{black}/{}/{$\blacktriangleleft$}/{}}};
		\draw pic (m5) at ({5 * \bS},0) {netA={{}/{}/{black}/{}/{}/{}}};
		
		\begin{scope}[very thin]
		\foreach \j in {1,...,5} {
			\pgfmathsetmacro \k { \j * \bS - 1 };
			\draw[gray,dashed] (\k,-3) -- (\k,-8);
			\draw[gray] (\k,3) -- (\k,-3);
		}
		\end{scope}
		
		\begin{scope}[shift={(0,-4)}, oriented WD, bbx = 1cm, bby =.4cm, bb min width=1cm, bb port sep=1]
		
		\draw node [fill=white,bb={1}{1}] (Tau) at (\bS - 1,0) {$\tau$};
		\draw node [fill=white,bb={1}{1}] (Mu)  at ({4 * \bS - 1},0) {$\mu$};
		\draw node [fill=white,bb={1}{1}] (Nu)  at ({3 * \bS - 1},{2 * \kY}) {$\nu$};
		
		\draw (-2,0) --     node[above] {$X$}       (0,0)
		--                  node[above] {}          (Tau_in1);
		
		\draw (Tau_out1) -- node[above] {$Y$}      ({2 * \bS - 1},0)
		--                  node[above] {$Y$}      ({3 * \bS - 1},0)
		--                  node[above] {$Y$}      (Mu_in1);
		
		\draw (Mu_out1) --  node[above] {$X$}      ({5 * \bS - 1},0)
		--                  node[above] {}         ({5 * \bS - 1},0)
		to[out=0, in=90]    node[above] {}         ({5 * \bS},-0.75)
		to[out=-90,in=0]    node[below] {}         ({5 * \bS - 1},{\kY})
		--                  node[above] {}         ({5 * \bS - 1},{\kY})
		--                  node[above] {$X^{-1}$} ({4 * \bS - 1},{\kY})
		--                  node[above] {$X^{-1}$} ({3 * \bS - 1},{\kY})
		--                  node[above] {$X^{-1}$} ({2 * \bS - 1},{\kY})
		to[out=-180,in=90]  node[above] {}         ({2 * \bS - 2},{2 * \kY - \hkY})
		to[out=-90,in=-180] node[above] {}         ({2 * \bS - 1},{2 * \kY})
		--                  node[above] {$X$}      (Nu_in1);
		
		\draw (Nu_out1) --  node[above] {$Z$}      ({4 * \bS - 1},{2 * \kY})
		--                  node[above] {$Z$}      ({5 * \bS - 1},{2 * \kY})
		--                  node[above] {}         ({6 * \bS - 1},{2 * \kY})
		--                  node[above] {$Z$}      ({6 * \bS},{2 * \kY});
		
		\end{scope}
		\end{tikzpicture}
	}
	\caption{In $\ExPetriZZ$, Conflict resolution. In $\ExPetriZ$, business strategy involving short-selling.}
	\label{fig:resolution}
\end{figure}

\noindent
One downside of this approach is that $U_1$ is acting on a token that ``was already there'', while $U_2$ is acting on a token produced on the fly by a cup. Clearly these two tokens are different, since the first may have a complex history (e.g. it was consumed/produced by other transitions before) while the second one is created ``on the fly''. The whole point is obviously that $U_2$ \emph{does not know} this: He is unaware of $U_1$'s action, and believes to be using $U_1$'s token! 

Clearly, to solve this issue we need a way to decide who acted first between $U_1$ and $U_2$. This is a typical \emph{consensus problem}, meaning that we need a way to establish an objective case (namely who acted first between $U_1$ and $U_2$) among a group of agents that may have different points of view. The advantage of our approach is that \emph{the amount of consensus required is small}: Instead of having to converge on what is the best way to solve the conflict (i.e. by merging transitions in some strange way, or dropping the execution of one of the two transitions altogether etc.), the agents have to agree on a very simple fact that for instance could be solved, implementation-wise, just using timestamps. The execution structure of the net will take care of the rest.

As usual, Figure~\ref{fig:resolution} has an interpretation also in the context of integer nets: In this case we again embrace the economics-oriented perspective. Places $X, Y, Z$ represent accounts (more specifically portfolios), belonging to different agents, that will be denoted with the same names $X, Y, Z$ to avoid clutter. Imagine that $X$ holds some financial instrument, and $X$ predicts that prices will fall sharply in the future: Figure~\ref{fig:resolution} may represent a possible business strategy. First, $X$ sells the instrument to $Y$. To capitalize further on his prediction, he also \emph{short sells}
\footnote{In finance, \emph{going short} means buying a financial instrument with the expectation that it will decrease in value~\cite{InvestopediaShort}. \emph{Short selling} is the act of selling an asset that the seller does not own, and includes borrowing the instrument from a broker.}
the financial instrument to $Z$, for instance via going long
\footnote{\emph{Going long} is the opposite of going short, and means buying a financial instrument with the expectation that it will acquire value~\cite{InvestopediaLong}.} 
on a put option
\footnote{A \emph{put option} is a contract giving the owner the right, but not the obligation, to sell a specified amount of the instrument at a specified price and time~\cite{InvestopediaPut}. Note that in this case $X$ is going long on the put option, because the more prices go down with regard to the selling value specified in the contract, the more the contract will acquire value.}. When prices go down, $X$ buys back the financial instrument from $Y$ and exercises the put option (effectively selling to $Z$), making a double profit. 

In this example cups and caps are helpful to represent the idea that a financial operation can be executed at a different time from its purchase. 
Note also how we can infer, from this diagram, that $Y$'s strategy consists in going long on the financial instrument owned by $X$: $Y$ hopes to sell the instrument back to $X$ at a profit, which is possible only if it acquires value in the short future.

\section{Conclusion and Future work}\label{sec:future work}
In this work we generalized the notion of Petri nets in two different ways: We built their categories of executions and demonstrated the power of our formalism graphically. We were driven by practical applications, namely a way to represent conflict resolution in Petri nets and a way to represent economic phenomena and accounting. Both areas of research are deeply entangled in the Statebox project. Statebox~\cite{Statebox} is a programming language for complex infrastructure entirely based on Petri nets, and runs in a decentralized way using Blockchain-based solutions. The Blockchain~\cite{Nakamoto2008} is needed exactly to establish consensus when the net is run by multiple users. 

\noindent
Blockchain-based consensus deals with conflicts by accepting only one state and discarding conflicting others. Semi-integer nets provide an alternative way to preserve consensus by merging conflicting states in a net. This is achieved reshuffling the causality flow of the executed transitions. The practical use and hopefully implementation of the concepts presented here will surely be object of future work. We will also work to expand the categorical machinery to deal with more sophisticated conflicting scenarios.

On the other hand, integer nets are useful to model economic flows. This is very interesting from a Blockchain perspective, where the usual way of representing resources (computations included) is by monetizing them~\cite{Buterin2014}. An operative account of economics is then very useful to design the best way to represent a given asset on the Blockchain, and to create high-level tools to solve long-standing open problems, such as smart contract analysis~\cite{Atzei2017}. Future work will include investigating applications of integer Petri nets to real economic phenomena, and how they relate to other established tools in the field, such as Open Games~\cite{Ghani2016}.

Finally, from a genuinely academic point of view, the structural similarities with the categorical framework for quantum mechanics -- namely compact closed categories -- are worth studying in depth, since the possibility of describing quantum phenomena by means of Petri nets cannot be excluded a priori. If our intellectual resources allow it, this will surely be another future direction of research.

\section*{Acknowledgements}
The authors would like to thank John Baez and Pawel Sobocinski, for having convinced them that the topic was worthy of being turned into a paper. They also thank David Spivak for having shared with them his invaluable knowledge and his TikZ macros, boosting their productivity by several orders of magnitude. Finally, they thank Emilia Gheorghe for having edited and proofread this document.

%
%

\bibliographystyle{eptcs}
\bibliography{bibtexbibliography}
\newpage
\appendix
\section{Proofs}
\structuralsubcategory*
\begin{proof}
	Suppose that $u$ is an object of $\CategoryC$ such that $\Multiplicity(u) = \nu$. Moreover, recall that by definition $u$ has the form $\bigtensor_{i=1}^n u_i$ with each $u_i \in S$. The claim is obvious from the following considerations:
	\begin{itemize}
		\item Identities do not change anything, so clearly $\id{u}{u} = u$; 
		\item $\Multiplicity$ is insensitive to ordering, all it does is counting. Hence applying a symmetry to $u$ doesn't change its image through $\Multiplicity$, meaning: $\Multiplicity(\sigma(u)) = \Multiplicity(u)$;
		\item Given an element $u_i$ of $S$, couples of the form $(u_i^{-1}, u_i)$ do not change $\Multiplicity(u)$, since $\Multiplicity$ is evaluated as $1$ on $u_i$ and as $-1$ on $u_i^{-1}$. This means that such couples can be added and subtracted at will anywhere in/from $u$;
		\item All the cups/caps in $\CategoryC$ can be obtained composing symmetries and cups/caps on the elements of $S$, and their inverses. This means that cups/caps always add/remove elements from $u$ in pairs as in the previous point, leaving $\Multiplicity(u)$ unaltered;
		\item A structural arrow is a composition of monoidal products of identities, symmetries and cups/caps. Since all these things do not alter $\Multiplicity(u)$, applying a structural arrow to $u \in \Gr{\CategoryC, \nu}$ gives us an object $v \in \Gr{\CategoryC, \nu}$. 
	\end{itemize}
	Being $u$ a generic element of $\Gr{\CategoryC, \nu}$ this proves that composition is well defined in $\Gr{\CategoryC, \nu}$. Associativity of morphisms is inherited from $\CategoryC$ and the existence of identities from the fact that identities are trivially structural arrows.
\end{proof}
\frakrisacongruence*
\begin{proof}	
	We start proving that $\mathfrak{R}_{\CategoryC, \CategoryD}$ is an equivalence relation for each couple of integer Petri categories $\CategoryC, \CategoryD$. Clearly for any functor $F:\CategoryC \to \CategoryD$ it is $(F,F) \in \mathfrak{R}_{\CategoryC, \CategoryD}$ since there is always an identity natural transformation $\id{F}: F \to F$. Moreover, if $(F,G) \in \mathfrak{R}_{\CategoryC, \CategoryD}$, then $(G,F) \in \mathfrak{R}_{\CategoryC, \CategoryD}$ since the definition of $\mathfrak{R}_{\CategoryC, \CategoryD}$ is totally symmetric (just invert the roles of $\tau$ and $\tau'$ in the definition). Now suppose that $(F,G)$ and $(G,H)$ are in $\mathfrak{R}_{\CategoryC, \CategoryD}$. Then there are natural transformations
	\begin{equation*}
		\tau: F \to G \qquad \tau': G \to F \qquad 	\bar{\tau}: G \to H \qquad 	\bar{\tau}':H \to G
	\end{equation*}
	With components all structural arrows. But then the componentwise compositions $\tau;\bar{\tau}: F \to H$ and $\bar{\tau}'\tau': H \to F$ have all components structural arrows (since composition of structural arrows is a structural arrow), proving that $(F,H) \in \mathfrak{R}_{\CategoryC, \CategoryD}$.
	
	Now let's prove that $\mathfrak{R}$ defines a congruence. This means that it respects compositions. Consider $(F,G) \in \mathfrak{R}_{\CategoryC, \CategoryD}$ and $(F',G') \in \mathfrak{R}_{\CategoryD, \CategoryE}$. As usual we have:
	\begin{equation*}
		\tau: F \to G \qquad \tau': G \to F \qquad 	\bar{\tau}: F' \to G' \qquad \bar{\tau}':G' \to F
	\end{equation*}
	So, the naturality conditions guarantee that, for each couple of objects $u,v \in \mathcal{C}$ and morphism $f:u \to v$, 
	\begin{equation*} 
	\begin{tikzpicture}[node distance = 2.5cm]
		\node (1) {$(F;F')u$};
		\node[right of = 1] (2) {$(F;F')v$};
		\node[below of = 1] (3) {$(G;F')u$};
		\node[below of = 2] (4) {$(G;F')v$};
		\node[below of = 3] (5) {$(G;G')u$};
		\node[below of = 4] (6) {$(G;G')v$};		
		
		\draw[->] (1) to node[midway, above] {$(F;F')f$}  (2);
		\draw[->] (1) to node[midway, left] {$F'\tau_{u}$}  (3);
		\draw[->] (2) -- (4) node[midway, right] {$F'\tau_{v}$};
		\draw[->] (3) -- (4) node[midway, above] {$(G;F')f$};
		\draw[->] (3) to node[midway, left] {$\bar{\tau}_{Gu}$}  (5);
		\draw[->] (4) -- (6) node[midway, right] {$\bar{\tau}_{Gv}$};
		\draw[->] (5) -- (6) node[midway, above] {$(G;G')f$};
	\end{tikzpicture}
	\qquad
	\begin{tikzpicture}[node distance = 2.5cm]
		\node (1) {$(G;G')u$};
		\node[right of = 1] (2) {$(G;G')v$};
		\node[below of = 1] (3) {$(F;G')u$};
		\node[below of = 2] (4) {$(F;G')v$};
		\node[below of = 3] (5) {$(F;F')u$};
		\node[below of = 4] (6) {$(F;F')v$};		
		
		\draw[->] (1) to node[midway, above] {$(G;G')f$}  (2);
		\draw[->] (1) to node[midway, left] {$G'\tau'_{u}$}  (3);
		\draw[->] (2) -- (4) node[midway, right] {$G'\tau'_{v}$};
		\draw[->] (3) -- (4) node[midway, above] {$(F;G')f$};
		\draw[->] (3) to node[midway, left] {$\bar{\tau}'_{Fu}$}  (5);
		\draw[->] (4) -- (6) node[midway, right] {$\bar{\tau}'_{Fv}$};
		\draw[->] (5) -- (6) node[midway, above] {$(F;F')f$};
	\end{tikzpicture}		
	\end{equation*}		
	Commute. Since both $F'$ and $G'$ are \underline{transition}-preserving they carry structural arrows to structural arrows, and hence the diagram on the left (resp. on the right) defines a natural transformation $F;F' \to G;G'$ (resp. $G;G' \to F;F'$) whose components are all structural arrows, proving $(F;F', G;G') \in \mathfrak{R}_{\CategoryC, \CategoryE}$.	

	\bigskip
	\noindent
	Now we prove the second claim. Consider a \underline{transition} component $\theta_{x,y}$, and structural arrows $\gamma: x'\to x$, $\gamma': y \to y'$, with $x, x' \in \Gr{\CategoryC, \nu}$ and $y,y' \in \Gr{\CategoryC, \nu'}$. Then, being $\theta$ a natural transformation, the following diagrams commute:
	
	\begin{equation*} 
		\begin{tikzpicture}[node distance = 2.5cm]
		\node (1) {$x'$};
		\node[right of = 1] (2) {$y$};
		\node[below of = 1] (3) {$x$};
		\node[below of = 2] (4) {$y$};
		
		\draw[->] (1) to node[midway, above] {$\theta_{x',y}$}  (2);
		\draw[->] (1) to node[midway, left] {$\gamma$}  (3);
		\draw[-] (2) -- (4) node[midway, left] {};
		\draw[->] (3) -- (4) node[midway, above] {$\theta_{x,y}$};
		\end{tikzpicture}
		\qquad
		\begin{tikzpicture}[node distance = 2.5cm]
		\node (1) {$x'$};
		\node[right of = 1] (2) {$y$};
		\node[below of = 1] (3) {$x'$};
		\node[below of = 2] (4) {$y'$};
		
		\draw[->] (1) to node[midway, above] {$\theta_{x',y}$}  (2);
		\draw[-] (1) to node[midway, left] {}  (3);
		\draw[->] (2) -- (4) node[midway, right] {$\gamma'$};
		\draw[->] (3) -- (4) node[midway, above] {$\theta_{x',y'}$};
		\end{tikzpicture}
	\end{equation*}	
	Proving $\gamma;\theta_{x,y};\gamma' = \theta_{x',y};\gamma' = \theta_{x',y'}$. From this it is easy to see that if $\theta_{Fu,Fv} = \theta'_{Fu,Fv}$, then $\theta = \theta'$. This is because given any two elements in $\Gr{\CategoryC, \nu}$ (in $\Gr{\CategoryC, \nu'}$, respectively) there is always a structural morphism in $\Gr{\CategoryC, \nu}$ (in $\Gr{\CategoryC, \nu'}$, respectively) connecting them, since each element can be obtained from another permuting it and adding/erasing pairs of elements $u$ and $u^{-1}$ using cups and caps.
	
	\bigskip
	\noindent
	The last claim is obvious.
%
\end{proof}
\maintheorem*
\begin{proof}
	First of all, note that the category $\Fold{N}$ is isomorphic to the category $\CategoryC$ having $\IStrings{P_N}$ as objects, and whose arrows are generated by the rules:
	\begin{gather*}
		\frac{u \in \IStrings{P_N}}{\id{u} \in  \Hom{\CategoryC}{u}{u}}
		\quad \frac{u \in \IStrings{P_N}}{\epsilon_u \in  \Hom{\CategoryC}{u \tensor u^{-1}}{I}} \quad \frac{u \in \IStrings{P_N}}{\eta_u \in  \Hom{\CategoryC}{I}{u \tensor u^{-1}}} \quad \frac{u,v \in \IStrings{P_N}}{\sigma_{u,v} \in  \Hom{\CategoryC}{u \tensor v}{v \tensor u}}\\\\
		\frac{t \in T_N}{t_{u,v}\in  \Hom{\CategoryC}{u}{v}} \quad \forall u,v.(\Multiplicity(u)= \Pin{t} \wedge \Multiplicity(v)= \Pout{t})\\\\
		\frac{\alpha \in  \Hom{\CategoryC}{u}{v}, \,\, \beta \in  \Hom{\CategoryC}{u'}{v'}}{\alpha \tensor \beta \in  \Hom{\CategoryC}{u \tensor u'}{v \tensor v'}} 
		\qquad 		
		\frac{\alpha \in  \Hom{\CategoryC}{u}{v}, \,\, \beta \in  \Hom{\CategoryC}{v}{w}}{\alpha ; \beta \in  \Hom{\CategoryC}{u}{w}}
	\end{gather*}
	Modulo the axioms that make it into a strict compact closed category:
	\begin{align*}
		\alpha ; \id{v} = &\alpha = \id{u} ; \alpha & \quad  (\alpha;\beta);\gamma &= \alpha;(\beta;\gamma)\\
		\tensorUnit\tensor \alpha = &\alpha = \alpha \tensor \tensorUnit& \quad  (\alpha \tensor \beta) \tensor \gamma &= \alpha \tensor (\beta \tensor \gamma)\\
		\id{u} \tensor \id{v} &= \id{u \tensor v} & \quad  (\alpha \tensor \alpha') ; (\beta \tensor \beta') &= (\alpha ; \beta) \tensor (\alpha' ; \beta')\\
		\sigma_{u, v \tensor w} = (\sigma_{u,v} &\tensor \id{w}); (\id{v} \tensor \sigma_{u,w}) & \quad  \sigma_{u,v};\sigma_{v,u} &= \id{u \tensor v}\\
		\sigma_{u,u'};(\beta \tensor \alpha) &= (\alpha \tensor \beta);\sigma_{v,v'} &\forall \alpha,\beta.(\alpha \in  \Hom{\CategoryC}{u}{v} &\wedge \beta \in  \Hom{\CategoryC}{u'}{v'}\\
		(\id{v} \tensor \eta_{v});(&\epsilon_{v} \tensor \id{v}) = \id{v} & \quad (\eta_{v} \tensor \id{v^{-1}});(&\id{v^{-1}} \tensor \epsilon_{v}) = \id{v^{-1}}\\
		\epsilon_{{u}^{-1}} &= \sigma_{u^{-1}, u} ; \epsilon_u &\quad \eta_{u^{-1}} &= \eta_u ; \sigma_{u^{-1}, u} 
	\end{align*}
	And the axioms:
	\begin{align*}
		\epsilon_{{u}^{-1}} &= \sigma_{u^{-1}, u} ; \epsilon_u  &\qquad \eta_u;\sigma_{u^{-1},u}&;\epsilon_{u^{-1}} = \id{\tensorUnit}\\ 
		p;t_{u',v'} &= t_{u,v};q &\qquad  \forall p, q. (p \in \Hom{\Gr{P_N}}{u}{u'}& \wedge q \in \Hom{\Gr{P_N}}{v}{v'})
	\end{align*}
	This is not difficult to prove using the freeness of $\Gr{P_N}$.
	
	We now prove that $\Fold{-}$ is a functor. Let $N \eqdef \Net{N}$ and $M \eqdef \Net{M}$ be nets, and $\langle f, g\rangle$ a morphism $N \to M$. We want to use the information given by $f$ and $g$ to build a functor between the categories $\Fold{N}$ and $\Fold{M}$.
	We sketch this procedure as follows:
	\begin{itemize}
		\item We note that $g$ is an homomorphism $\Iset{P_N} \to \Iset{P_M}$ of free abelian groups, and hence corresponds uniquely to a function $g':P_N \to \Iset{P_M}$ because of the freeness of $\Iset{P_N}$.
		\item We need a way to lift $g$ to an equivalence class of \underline{transition}-preserving functors $\Fold{N} \to \Fold{M}$. As an intermediate step, we start lifting $g'$ to a functor $\Gr{P_N} \to \Fold{M}$.
		\item To do this, note that for each right inverse of $\Multiplicity$, viz. some $\alpha:\Iset{P_M} \to \IStrings{P_M}$ such that $\Multiplicity(\alpha(u)) = u$, the composition $g';\alpha: P_N \to \IStrings{P_M}$ is a function from $P_N$ to $\obj{\Fold{M}}$. 
		\item $\Fold{M}$ is clearly strict compact closed and by definition respects the axiom $\eta_u;\sigma_{u^{-1},u};\epsilon_{u^{-1}} = \id{\tensorUnit}$, hence we can use Remark~\ref{rem:freeness of gr} to obtain a functor $F:\Gr{P_N} \to \Fold{M}$.
		\item We extend the functor $F:\Gr{P_N} \to \Fold{M}$ to a functor $\Fold{\langle f,g \rangle}: \Fold{N} \to \Fold{M}$, mapping $t_{u,v}$ to $f(t)_{F(u),F(v)}$ and coinciding with $F$ on cups, caps and symmetries.
		\item Using the equivalent axiomatization for $\Fold{N}$ presented above, we note that the ${t_{u,v}}$ are the components of a \underline{transition} in $\Fold{N}$, and that all \underline{transitions} in $\Fold{N}$ have components of type $t_{u,v}$ for some $t \in T_N$ and objects $u,v$. On the other hand, ${f(t)_{Fu, Fv}}$ are the components of a \underline{transition} in $\Fold{M}$, and hence $\Fold{\langle f,g \rangle}$ preserves \underline{transitions}. 
		\item For each suitable choice of $\alpha$, we get functors that are identified by the relation $\mathfrak{R}_{\Fold{N},\Fold{M}}$, hence this construction is independent from the choice of $\alpha$. Moreover, if $F, G$ are two functors $\Fold{N} \to \Fold{M}$ and $(F,G) \in \mathfrak{R}_{\Fold{N},\Fold{M}}$ then both $F,G$ are obtained using the procedure sketched above, with different choices of $\alpha$.
	\end{itemize}
	This ensures that what we are doing is well defined , and that to each $\langle f, g \rangle$ corresponds a morphism $\Fold{\langle f,g \rangle}$ in the category $\GExPetri$. Proving functoriality of $\Fold{-}$ is easy noting that the identity morphism $\langle \id{T_N}, \id{\Iset{P_N}} \rangle$ is mapped in the equivalence class of the identity functor $\Fold{N} \to \Fold{N}$, and hence to the identity functor on $\Fold{N}$ in $\GExPetri$. Composition is preserved noting that $\Fold{\langle f;f', g;g'\rangle }$ sends a transition $t$ to $(f;f')(t)$, that is a \underline{transition} since composition of \underline{transition}-preserving functors is \underline{transition}-preserving, and $\mathfrak{R}$ is a congruence with respect to this.
	
	\bigskip
	\noindent
	Next step is to prove functoriality of $\UnFold{-}$. First of all we have to determine how $\UnFold{-}$ acts on morphisms. Let $\CategoryC$ and $\CategoryD$ be integer Petri categories. We denote their monoids of objects as $\IStrings{C}$ and $\IStrings{D}$, respectively. If $F:\CategoryC \to \CategoryD$ is a representative of a morphism in $\GExPetri$, then by definition it preserves structural arrows, meaning that if $u,v \in \Gr{\CategoryC, \nu}$ then $\Multiplicity(Fu)=\Multiplicity(Fv)$.
	But then the action of $F$ on objects induces an obvious homomorphism of free abelian groups $F_{pl}:\Iset{C} \to \Iset{D}$, which is clearly independent from the choice of representative $F$. 
	
	Since $F$ is also \underline{transition}-preserving, then it induces an obvious function $F_{tr}: T_{\UnFold{\CategoryC}} \to T_{\UnFold{\CategoryD}}$: Each transition $t$ of the net $\UnFold{\CategoryC}$ comes by definition from a \underline{transition} $\tau$ in $\CategoryC$. $F_{tr}$ then maps each $t$  to the transition of $\UnFold{\CategoryD}$ coming from the \underline{transition} in $\CategoryD$ having components $F\tau_{Fu,Fv}$. This is again clearly independent from the choice of $F$. We then set $\UnFold{F} := \langle F_{tr}, F_{pl} \rangle$. Functoriality at this point follows trivially from the definitions.

	\bigskip
	\noindent
	Now we have to prove the adjunction. We define the co-unit  $\epsilon:\Fold{\UnFold{-}} \to (-)$ specifying its components. For each $\CategoryC \in \obj{\GExPetri}$, the functor $\epsilon_\CategoryC:\Fold{\UnFold{\CategoryC}} \to \CategoryC$ is identity on objects and structural arrows. Given a \underline{transition} $\tau: \pi_{\CategoryC, \nu} \to \pi_{\CategoryC, \nu'}$ in the category $\CategoryC$, this will correspond to a transition $[\tau]$ in the net $\UnFold{\CategoryC}$, and this transition will be again mapped to a family of morphisms $[[\tau]]_{u,v}$ in $\Fold{\UnFold{\CategoryC}}$, with
	\begin{equation*}
		\Multiplicity(u) = \Pin{[\tau]} = \nu \qquad \Multiplicity(v) = \Pout{[\tau]} = \nu'
	\end{equation*}
	Thanks to the axioms in~\ref{eq: axiom}, the $[[\tau]]_{u,v}$ define a \underline{transition} in the category $\Fold{\UnFold{\CategoryC}}$, and we can define $\epsilon_\CategoryC$ as mapping each morphism $[[\tau]]_{u,v}$ to $\tau_{u,v}$. This obviously makes $\epsilon_\CategoryC$ \underline{transition}-preserving, and thus a representative of a morphism in $\GExPetri$. Then each $[\epsilon_\CategoryC]_{\mathfrak{R}_{\CategoryC, \CategoryC}}$ (here $[-]$ denotes an equivalence class) defines the components of a natrual transformation $\epsilon:\Fold{\UnFold{-}} \to (-)$, and proving that $\epsilon$ has the co-universal property is just a straightforward check.
	
	The unit $\eta: (-) \to \UnFold{\Fold{-}}$ is much easier to define: For each net $N \eqdef \Net{N}$, $\eta_N$ it is the isomorphism $\langle \varphi, id_{\Iset{P_N}} \rangle$, where $\varphi$ is the bijection sending each transition $t$ of $N$ to the transition $[t]$ of $\UnFold{\Fold{N}}$ coming from the \underline{transition} in $\Fold{N}$ having components $t_{u,v}$. 
\end{proof}
\firstequivalence*
\begin{proof}
	We just have to show that the counit $\epsilon_\CategoryC$ of the adjunction $\Fold{-} \vdash \UnFold{-}$ is an iso if and only if $\CategoryC \in \PetriZ$, which is obvious from the definitions.
\end{proof}
\secondequivalence*
\begin{proof}
	This is quite easy. First of all we note that $\PetriZZ$ can obviously be considered a full subcategory of $\PetriZ$. This is clear since any multiset can also be considered as a signed multiset, and homomorphisms of free commutative monoids $\Mset{S} \to \Mset{S'}$ can be lifted to a homomorphisms of free abelian groups $\Iset{S} \to \Iset{S'}$ via the usual free properties.
	The equivalence then follows trivially that:
	\begin{itemize}
		\item The image through $\Fold{-}$ of a net in $\PetriZZ$ is a positive object in $\ExPetriZ$, which is obvious from the definition;
		\item The image through $\UnFold{-}$ of a positive object in $\ExPetriZ$ is a net in $\PetriZZ$, which is again obvious from the definition.
	\end{itemize}
\end{proof}\label{sec:appendix}

\end{document}